\documentclass[final,1p,times]{elsarticleHF}

\usepackage[english]{babel}
\usepackage[utf8]{inputenc}
\usepackage[T1]{fontenc}
\usepackage{lmodern}
\usepackage[locale=UK]{siunitx}
\usepackage{pgfplots}
\pgfplotsset{cycle list name=black white,
	axis x line=bottom,
	axis y line=left,
	xlabel near ticks,
	ylabel near ticks,
	clip mode=individual,
	/pgf/number format/use comma,
	width= .8\textwidth,
	height=7cm,
	every axis plot/.append style={line join=round},
	}
\pgfplotsset{compat=1.9}
\usepackage{pgfplotstable}
\usepackage{xfrac}
\usepackage{graphicx}
\usepackage{xcolor}
\usepackage{import}
\usepackage{todonotes}
\usepackage{bbm}

\usepackage{showkeys} 
\usepackage{amsmath, amssymb, amsthm} 
\usepackage{mathtools} 
\usepackage{algorithm}
\usepackage{algpseudocode}
\algnewcommand\algorithmicinput{\textbf{Input:}}
\algnewcommand\Input{\item[\algorithmicinput]}
\algnewcommand\algorithmicoutput{\textbf{Output:}}
\algnewcommand\Output{\item[\algorithmicoutput]}
\numberwithin{equation}{section}
\usepackage{hyperref}	
\usepackage[english]{cleveref} 
\usepackage{autonum}
\usepackage{dsfont} 

\newtheorem{thm}{Theorem}
\newtheorem{rem}[thm]{Remark}

\theoremstyle{definition}

\DeclareMathOperator{\diag}{diag}
\DeclareMathOperator{\Real}{Re}
\DeclareMathOperator{\Imag}{Im}
\DeclareMathOperator{\vect}{vec}

\DeclareMathOperator{\sgn}{sgn}

\begin{document}
\begin{frontmatter}

\title{Lyapunov and Sylvester equations:\\A quadrature framework}

\author[label1]{Christian Bertram\corref{cor1}}
\author[label1]{Heike Fa\ss bender}
\address[label1]{Institut \emph{Computational Mathematics}/ AG Numerik, TU Braunschweig, Universitätsplatz 2, 38106 Braunschweig, Germany}
\cortext[cor1]{Corresponding author, Email ch.bertram@tu-braunschweig.de}

\begin{abstract}
	This paper introduces a novel framework for the solution of (large-scale) Lyapunov and Sylvester equations derived from numerical integration methods. Suitable systems of ordinary differential equations are introduced. Low-rank approximations of their solutions are produced by Runge-Kutta methods. Appropriate Runge-Kutta methods are identified  following the idea of geometric numerical integration to preserve a geometric property, namely a low rank residual. For both types of equations we prove the equivalence of one particular instance of the resulting algorithm to the well known ADI iteration. As the general approach suggested here leads to complex valued computation even for real problems, we present a general realification approach based on similarity transformation.
\end{abstract}


\end{frontmatter}

\section{Introduction}
The numerical approximation of the solution of the continuous Lyapunov equation
\begin{align}
	\label{eq:lyap}
	A\mathcal{P} + \mathcal{P}A^\mathsf{T} + BB^\mathsf{T} &= 0
\end{align}
has been considered to great extend in the literature, see, e.g. the recent survey \cite{simoncini2016} and the references therein. Here
a new framework based on methods for the numerical integration of ordinary differential equations (ODEs) 
is presented. We will consider \eqref{eq:lyap} for  $A\in \mathbb{R}^{n\times n}$ and $B\in \mathbb{R}^{n\times m}$, $m \leq n.$
The equation \eqref{eq:lyap} has a unique symmetric positive definite solution $\mathcal{P}$
if all eigenvalues of $A$ are in the open left half-plane $\mathbb{C}_{-}$ (that is, if $A$ is stable). In that case, the analytic solution can 
be written as
\begin{align}\label{eq:P}
	\mathcal{P} &= \int_0^\infty e^{At}BB^{\mathsf{T}}e^{A^{\mathsf{T}}t}\,\mathrm dt \in \mathbb{R}^{n \times n},
\end{align}
see, e.g., \cite{LanR95}. 
Lyapunov equations play an important role in control and systems theory, see, e.g. \cite{Ant2005,Dat94,GajQ95}. 
In the context of linear time-invariant systems $\dot{x} = Ax + Bu,$ the solution $\mathcal{P}$ of \eqref{eq:lyap} is called the controllability Gramian. It measures the energy transfer in the system.

We will also consider Sylvester equations
\begin{align}
	\label{eq:sylv}
	A\mathcal{Y} - \mathcal{Y}B - FG^{\mathsf{T}} &= 0
\end{align}
with given $A\in \mathbb{R}^{n \times n},$ $B\in \mathbb{R}^{m\times m},$ $F\in \mathbb{R}^{n\times r}$ and $G\in \mathbb{R}^{m\times r}$. The solution $\mathcal{Y}$ is unique when the spectra of $A$ and $B$ are disjunct, i.e. $\sigma(A)\cap\sigma(B)=\emptyset,$ see, e.g. \cite{LanR95}.

In particular, we will be concerned with \eqref{eq:lyap} and \eqref{eq:sylv} for
large and sparse system matrices and a low rank initial residual, that is for the Lyapunov equation $BB^{\mathsf{T}}$ and for the Sylvester equation $FG^{\mathsf{T}}$ is of low rank. In case of \eqref{eq:lyap} $m \ll n$ for large $n$ will automatically
yield a low rank residual. In that case, the symmetric positive definite solution $\mathcal{P}$ of \eqref{eq:lyap} can be approximated by a low rank approximation in the sense that $\mathcal{P}\approx Z Z^{\mathsf{T}}$ with a rectangular 
$n\times N$ matrix $Z, N \ll n,$ \cite{li2002,benner2013}. $Z$ is often called a low rank Cholesky factor, even so $Z$ is not a
square lower triangular matrix.
In a similar fashion, if $FG^{\mathsf{T}}$ is of low rank, the solution of the Sylvester equation can be approximated by $\mathcal{Y}\approx \hat Z\Gamma\breve Z^{\mathsf{T}}$ with  rectangular $n\times N$ matrices $\hat Z,\, \breve Z, N\ll n$  and a diagonal matrix $\Gamma,$ see, e.g. \cite{benner2009,grasedyck2004}.

In the following we give an overview of methods important or related to our later discussion. For a more exhaustive survey of methods for the solution of various linear matrix equations we refer to \cite{simoncini2016}. 

A popular algorithm for deriving low rank Cholesky factors for Lyapunov and Sylvester equations is the alternate directions implicit (ADI) method. It was developed to solve linear systems of equations in \cite{peaceman55} and modified to approximate the solution $\mathcal{P}$ of Lyapunov equations in \cite{lu91} (see \cite[Sec. 7]{benbreifen2015} for a short or \cite[Chp. 3.1-3.2.1]{Kue16} for a more detailed derivation). 
In \cite{li2002} the iteration was reformulated such that low rank approximations $Z_jZ_j^{\mathsf{T}}$ to the solution $\mathcal{P}$ are generated. This yields the computationally more efficient ADI-variant called Cholesky factor ADI (CF-ADI) algorithm.
Clearly, an approximate solution $P_j$  will not satisfy \eqref{eq:lyap} exactly, a nonzero residual $\mathcal{L}(P_j) =AP_j+P_jA^{\mathsf{T}} + BB^{\mathsf{T}}$ will remain. This residual can be used to determine convergence of the iterative process. In  \cite[Alg. 3.2]{Kue16}, the ADI iteration was further manipulated in order to allow for a fast evaluation of the residual norm $\|\mathcal{L}(P_j)\|.$ This is known as the residual-based ADI method.
An alternative derivation of this formulation utilizing Krylov subspaces can be found in \cite{wolf13}. The ADI iteration was also adapted to Sylvester equations, see \cite{benner2009}, \cite[Chp. 3.3]{Kue16}.

Another type of methods for the solution of Lyapunov equations is making use of empirical Gramians  \cite{Moo81}. The empirical Gramian essentially involves  a sum approximation of the integral \eqref{eq:P}
$\mathcal{P} = \sum_j  \delta_j g(t_j)$ for $g(t) = e^{At}BB^{\mathsf{T}}e^{A^{\mathsf{T}}t},$ arbitrary times $t_j$
and appropriate quadrature weights $\delta_j.$ Usually, the identity $g(t)= h(t)h(t)^{\mathsf{T}}$ for $h(t)=e^{At}B$
is used to determine $g(t_j).$ In doing so, $h(t)$ is not computed directly, but as the solution of the ordinary differential equation (ODE) $\frac{\mathrm d}{\mathrm dt}h(t)=Ah(t)$ with initial value $h(0)=B.$ Any numerical integration scheme can be used to do so.
Related quadrature approaches are discussed, e.g., in \cite{saad90, singler11}. Empirical Gramians are mainly used in 
model order reduction via balanced proper orthogonal decomposition (POD), see \cite{Rowley05}.

The ADI iteration and the quadrature-based methods have been connected to rational Krylov subspaces and moment matching, see \cite{druskin11, Opm12, wolf13}.
In \cite{Opm12} quadrature methods with complex time stepsizes {for the approximation of empirical Gramians} were analyzed. Further the stability function of certain multi-stage implicit methods was connected to the (complex) interpolation points used in rational interpolation.


In this paper for the Lyapunov case we utilize the time-dependent Gramian 
\[
P(t) = \int_0^te^{A\tau} BB^\mathsf{T}e^{A^\mathsf{T}\tau} d\tau.
\] 
For $t \rightarrow \infty,$ $P(t)$ will approximate the solution $\mathcal{P}$ of the Lyapunov equation.
We will make use of the fact that $P(t)$ can be interpreted as the solution of a certain system of ODEs. It turns out that also in the context of Sylvester equations we can state a useful system of ODEs.  
 Runge-Kutta methods are employed to derive algorithms for the low rank approximation of the solution of Lyapunov and Sylvester equations. 
In the Lyapunov case, neither $P(t)$ nor the iterates $P_j$ from the Runge-Kutta methods will exactly satisfy the Lyapunov equation. We will observe that $\mathcal{L}(P(t)) = AP(t)+P(t)A^\mathsf{T}+BB^\mathsf{T}=h(t)h(t)^\mathsf{T}$ for $h(t) = e^{At}B.$ Our key idea is to use only those Runge-Kutta methods which lead to iterates $P_j$ with conformable low rank Lyapunov residuals.
Thus, we use ideas from geometric numerical integration, where qualitative properties (e.g. algebraic invariants) of the solution are preserved instead of fulfilling quantitative properties (e.g. small errors), cf. \cite[Chp. 5]{iserles_2008}, \cite{HaiLW06}. Herewith we derive a residual based iteration which turns out to be equivalent to the ADI iteration.
By making use of the stability function of a Runge-Kutta method it will be shown further that these methods are equivalent to DIRK methods.

The paper is organized as follows. In the next section we introduce notation, review basic properties of Runge-Kutta methods for numerical integration and give a short introduction to the ADI iteration for the solution of Lyapunov equations.
\Cref{sec:gramian_as_ode} deals with the Lyapunov equation and its numerical solution. In Section \ref{sec31} we present a first algorithm for computing a low rank approximation to the time-dependent Gramian $P(t)$ by means of a Runge-Kutta method.
Clearly, $P(t)$ will not exactly satisfy the Lyapunov equation. In Section \ref{sec:invariant},
we derive an expression for the Lyapunov residual $\mathcal{L}(P(T)).$ This turns out to be of low rank for $m \leq n.$ We propose to use only those Runge-Kutta methods which yield iterates  satisfying the same kind of Lyapunov residual as $P(t).$
Conditions for appropriate Runge-Kutta methods leading to such iterates are given. 
In Section \ref{sec33} the usual approach of using the same Runge-Kutta method in each iteration step is relaxed in order to
allow for the use of different Runge-Kutta methods in each iteration step.
Section \ref{sec34} deals with the equivalence of a certain instance of the resulting algorithm to the CF-ADI iteration. 
Next, in Section \ref{subsec_multupdate} the appropriate Runge-Kutta methods are further characterized  by means of their stability functions. 
As discussed in Section \ref{sec36}, it turns out, that the appropriate methods are essentially determined by $s$ parameters.
Finally, in Section \ref{sec:geometric}  the choice of (complex-valued) shifts is discussed, while in Section \ref{sec37} the realification of the
potentially complex arithmetic involving algorithm is considered.
The ideas for the solution of Lyapunov equations are transferred to the Sylvester equation in Section \ref{sec:sylvester}.
The paper ends with some concluding remarks in Section \ref{sec:conclusion}.
\section{Preliminaries}
\label{sec:preliminaries}
In this section we will introduce some notation used in the following as well as briefly recall Runge-Kutta methods for the numerical integration of ordinary differential equations. Moreover, the ADI method for solving Lyapunov equations \eqref{eq:lyap} is reviewed.

The set of complex numbers with positive (negative) real part will be denoted by $\mathbb{C}_{+}$ ($\mathbb{C}_{-}$). The positive (negative) real numbers will be denoted by $\mathbb{R}_{+}$ ($\mathbb{R}_{-}$). 

We will frequently make use of the Kronecker product of two matrices as well as the vectorization of a matrix, see, e.g.,
\cite{HorJ91,GolVL13} for a more complete discussion.
If $X$ is an $r \times s$ matrix and $Y$ is a $p \times q$ matrix, 
then the Kronecker product $X \otimes Y$ is the $rp \times sq$ block matrix
\[
    X \otimes Y  =\begin{bmatrix}
x_{11}Y &\cdots &x_{1s}Y \\\vdots &\ddots &\vdots \\x_{r1}Y &\cdots &x_{rs}Y \end{bmatrix}.
\]
If $X$ and $Y$ are regular, then the property
\begin{align}
	(X\otimes Y)^{-1} = X^{-1} \otimes Y^{-1}
\end{align}
holds.
Other useful Kronecker product properties are
\begin{align}\label{eq:kronmult}
(X \otimes Y)(V\otimes W) &= XV \otimes YW\\
P(X\otimes Y) Q^\mathsf{T} &= Y \otimes X\label{eq:kronswap}
\end{align}
for suitable $V, W$ and the perfect shuffle permutation matrices $P$ and $Q,$ see, e.g. \cite[Chp. 1.3.6]{GolVL13}. 
In particular, for square matrices $X$ and $Y$ (that is, $r=s$ and $p = q$), we have $P = Q = I_{rp}([(1{:}p{:}rp)\ (2{:}p{:}rp)\ \ldots\ (p{:}p{:}rp)],{:})$ where $I_{rp}$ denotes the $rp \times rp$ identity matrix and MATLAB\textsuperscript{\textregistered} colon notation is used to specify the arrangement of the rows of $I_{rp}.$ 
In case $X$ or $Y$ is a vector, the corresponding perfect shuffle permutation matrix is the identity: let $s=1$, then $P(X\otimes Y)I_{q} = Y \otimes X$. 

The vectorization of a matrix converts the matrix into a column vector. For a $r \times s$ matrix $X$, $\vect(X)$
denotes the $rs \times 1$ column vector obtained by stacking the columns of the matrix $X$ on top of one another:
\[
   \vect (X)=[x_{11},\ldots ,x_{r1},x_{12},\ldots ,x_{r2},\ldots ,x_{1s},\ldots ,x_{rs}]^\mathsf{T}.
\]
The vectorization and the Kronecker product can be used to express matrix multiplication as a linear transformation 
on matrices. In particular,
\[
  \vect(XYZ)=(Z^\mathsf{T} \otimes X) \vect(Y) 
\]
for matrices $X$, $Y$, and $Z$ of dimensions $r\times s$, $s\times t$, and $t\times v$. In particular,
 we can rewrite the Sylvester equation \eqref{eq:sylv}
in the form
\[
	(I_{n}\otimes A-B^\mathsf{T}\otimes I_{n})\vect(\mathcal{Y})=\vect(FG^{\mathsf{T}}),
\]
where  $I_{n}$ is the $n \times n$ identity matrix. In this form, the equation can be seen as a linear system of 
equations of dimension  $n^2\times n^2.$ From this it is fairly straightforward to see that a unique solution 
$\mathcal{Y}$ exists for all $FG^{\mathsf{T}}$ if and only if $A$ and $B$ have no common eigenvalues, see, e.g., \cite{HorJ91,LanR95}.

\subsection{Numerical integration}\label{sec:RK}
There are numerous methods for the numerical solution of ordinary differential equations of the type
\begin{align}
	\label{eq:ode}
	\frac{\mathrm d}{\mathrm dt}y(t)=f(t,y(t)), \quad y(0)=y_0,
\end{align}
see, e.g., \cite{HaiNW93,HaiW96}. Here $f\colon \mathbb{R}\times \mathbb{R}^n \rightarrow \mathbb{R}^{n}$ is a given function, $y_0 \in\mathbb{R}^{n}$ is a given initial value and one is interested in computing the function $y\colon \mathbb{R} \rightarrow \mathbb{R}^n$ in the interval $[0, t_\text{end}].$
Single-step methods make use of the fact that
\begin{align}
	y(t_j)&=y(t_{j-1})+\int_{t_{j-1}}^{t_j} f(t,y(t))\,\mathrm dt
\end{align}
holds for $j = 1, 2, \ldots, N$ and $ t_0 = 0 < t_1 < t_2 < \cdots < t_N = t_\text{end}$ in order to compute 
approximate solutions $y_j \approx y(t_j).$

We will consider $s$-stage Runge-Kutta methods (see, e.g., \cite{But16,HaiLW06,HaiNW93,HaiW96}) which are defined via
	\begin{align}
		\label{eq:runge-kutta}
		y_j &= y_{j-1} + \omega_j\sum_{i=1}^{s}\beta_ik_i^{(j)}, & j &= 1, \ldots, N,\\
		\label{eq:runge-kutta-k}
		k_i^{(j)} &= f\Big(t_{j-1}+\gamma_i\omega_{j},\, y_{j-1}+\omega_j\sum_{\ell=1}^{s}\lambda_{i\ell}k_\ell^{(j)}\Big),\ & i&=1,\dots,s,
	\end{align}
for certain $\beta_i \in \mathbb{C},$ $\gamma_i  \in \mathbb{R}$, $i = 1, \ldots, s$  and $\lambda_{i\ell} \in \mathbb{C}$, $i,\ell = 1, \ldots, s.$ 
Please note that we allow for complex-valued $\lambda_{ij}$ and $\beta_i$ unlike the standard definition of Runge-Kutta methods for the solution of \eqref{eq:ode}.
Moreover,  $\omega_j\coloneqq t_j-t_{j-1} > 0$, $j = 1, \ldots, N,$ denotes the time step size. Often Runge-Kutta methods are
given in short hand by the so called Butcher tableau
\begin{align}
	\label{eq:butcher}
\begin{array}{c|c}
\gamma& \Lambda\\
\hline
& \beta^{\mathsf{T}} \\
\end{array}=
\begin{array}{c|cccc}
\gamma_1    & \lambda_{11} & \lambda_{12}& \dots & \lambda_{1s}\\
\gamma_2    & \lambda_{21} & \lambda_{22}& \dots & \lambda_{2s}\\
\vdots & \vdots & \vdots& \ddots& \vdots\\
\gamma_s    & \lambda_{s1} & \lambda_{s2}& \dots & \lambda_{ss} \\
\hline
       & \beta_1    & \beta_2   & \dots & \beta_s\\
\end{array}
\end{align}
with $\Lambda\in\mathbb{C}^{s\times s}$, $\beta \in \mathbb{C}^s$ and $\gamma\in \mathbb{R}^{s}.$ 

If in the Butcher tableau $\Lambda$ is a strict lower triangular matrix, then the $k_i^{(j)}$ can be calculated 
explicitly one after another. Otherwise they are only defined implicitly and a system of (in general nonlinear) 
equations with $sn$ unknowns has to be solved to obtain them. For explicit Runge-Kutta methods the region of 
absolute stability is small and bounded. On the other hand, implicit Runge-Kutta methods have much larger regions of absolute stability
and the time step size can be chosen based on the desired accuracy, not due to stability constraints. In order
to avoid the high computational costs for general implicit Runge-Kutta methods, often so-called
diagonal implicit Runge-Kutta (DIRK) methods are used, where $\Lambda$ is a lower triangular matrix \cite{KenC16}. This uncouples the system of equations to be solved into a sequence of $s$ systems.

The function 
\begin{equation}\label{eq:stab}
	R(z) = 1 + z\beta^\mathsf{T}(I_s-z\Lambda)^{-1}\mathds{1}_s = \frac{\det(I_s-z(\Lambda-\mathds{1}_s\beta^\mathsf{T}))}{\det(I_s-z\Lambda)}
\end{equation}
is called the stability function of the Runge-Kutta method given by \eqref{eq:butcher}. Here, $\mathds{1}_s$ denotes the vector $\mathds{1}_s=(1,\dots, 1)^{\mathsf{T}}$ consisting of $s$ ones.
When a Runge-Kutta method is applied to the linear differential equation $y'=\lambda y$ the iteration is given by $y_{k}=R(z)y_{k-1}$ with $z=\omega\lambda$.
The method is said to be A-stable if all $z$ with Re$(z) < 0$ are in the domain of absolute stability, that is
the set of all $z=\omega\lambda$ with $|R(z)|<1.$

\subsection{The ADI method}\label{sec:adi}
Here we introduce the ADI iteration for the Lyapunov equation based on \cite{Kue16,benkuebt2013}. The goal is to approximate the $n \times n$ solution $\mathcal{P}$ of the Lyapunov equation \eqref{eq:lyap} in factored form $Z Z^{\mathsf{T}}\approx \mathcal{P}$, where $Z\in \mathbb{R}^{n\times mN}$ with $mN\ll n$. 

The ADI iteration for computing the solution of \eqref{eq:lyap} is given by
\begin{align}
	X_0 &= 0,\\
	(A+\alpha_j I_n)X_{j-\frac{1}{2}} &= -BB^{\mathsf{T}} - X_{j-1}(A^{\mathsf{T}}-\alpha_j I_n),\\
	(A+\alpha_j I_n)X_{j}^{\mathsf{H}} &= -BB^{\mathsf{T}} - X_{j-\frac{1}{2}}^{\mathsf{H}}(A^{\mathsf{T}}-\alpha_j I_n),
\end{align}
with complex shift parameters $\alpha_1, \dots, \alpha_N\in \mathbb{C}_{-}$. For a certain choice of shift parameters the iterates $X_j$ will converge to $\mathcal{P}.$
Reformulating this iteration into a single step, writing the iterates $X_j=Z_jZ_j^{\mathsf{H}}$ in factored form and applying some algebraic manipulations (see \cite[Chp. 3.2]{Kue16} for the details) one obtains the iteration
\begin{align}
	V_1 &= (A+\alpha_1 I_n)^{-1}B,\ Z_1 = \sqrt{-2\Real(\alpha_1)}V_1,\\
	V_j &= V_{j-1} - (\alpha_j+\overline{\alpha_{j-1}})(A+\alpha_jI_n)^{-1}V_{j-1},\\
	Z_j &= \left[ Z_{j-1}, \sqrt{-2\Real(\alpha_j)}V_j \right].
\end{align}
With the findings from \cite[Chp. 3.2.4]{Kue16} respectively, the iteration results in \Cref{alg:adi32}.
Please note that $Z_j$ grows in each iteration step by a block of $m$ columns.

\begin{algorithm}[ht]
    \caption{Low rank ADI iteration \cite[Alg. 3.2, $E=I_n$]{Kue16}}
    \label{alg:adi32}
    \begin{algorithmic}[1] 
	    \Input $A \in \mathbb{R}^{n \times n}$ stable, $B \in \mathbb{R}^{n\times m}$, parameters $\left\{ \alpha_1,\dots,\alpha_{N} \right\}\in \mathbb{C}_{-}$
	    \Output  $Z\in \mathbb{C}^{n\times mN}$ with $Z Z^{\mathsf{H}}\approx \mathcal{P}$
    \State initialize $W_0=B$, $Z_0=[\ ]$
    \For {$j=1,\dots,N$}
    \State solve $(A+\alpha_jI_n)V_j=W_{j-1}$ for $V_j$
    \State $W_j = W_{j-1}-2\Real(\alpha_j)V_j$
    \State update 
    $Z_j=[Z_{j-1}, \sqrt{-2\Real(\alpha_j)}V_j]$
    \EndFor
    \State $Z=Z_N$
    \end{algorithmic}
\end{algorithm}

As the shift parameters $\alpha_j$ are complex numbers, the iterates $V_j$ (and thus $Z_j$) are complex-valued matrices. 
This can be avoided if the set of shift parameters is proper, i.e. closed under complex conjugation such that complex shift parameters appear in pairs with their complex conjugated version. 

We will briefly present the realification approach from \cite{Benner2012} in the revised form of \cite[Chp. 4.1.4]{Kue16}. It is denoted $\mathsf{M4}$ in \cite[Chp. 4.1.5]{Kue16}.
Let $\alpha_j$ be a complex shift and $\alpha_{j+1}= \overline{\alpha_j}$. Instead of performing two separate steps,
one with the shift $\alpha_j$ and one with $\alpha_{j+1},$ a double step involving $\alpha_j$ and $\alpha_{j+1}$ will be used.
That is, the block $\hat{Z} =\sqrt{-2\Real(\alpha_j)}\begin{bmatrix}V_j& V_{j+1}\end{bmatrix}$ of $2m$ columns is added to the current iterate $Z_{j-1}.$ This is still a complex matrix, but it can be replaced by a real one.
For this, note that
\begin{align}
	V_{j+1} &= \overline{V}_j + 2\frac{\Real(\alpha_j)}{\Imag(\alpha_j)}\Imag(V_j).
\end{align}
Let
\begin{equation}
	J=\begin{bmatrix} 1 & 1\\ {\imath}&2\frac{\Real(\alpha_j)}{\Imag(\alpha_j)}-{\imath} \end{bmatrix} \text{ and } L=\sqrt{2}\begin{bmatrix}1& 0\\ \frac{\Real(\alpha_j)}{\Imag(\alpha_j)}& \sqrt{\frac{\Real(\alpha_j)^2}{\Imag(\alpha_j)^2}+1}\end{bmatrix}.\label{eq:L}
\end{equation}
Then the block $\hat{Z}$ of $2m$ columns that is added to the factor $Z_{j-1}$ can be written as
\begin{align}
\hat{Z} 
&= \sqrt{-2\Real(\alpha_j)}\begin{bmatrix}\Real(V_j)\ \Imag(V_j) \end{bmatrix}\left(J \otimes I_m\right).
\end{align}
Observing that
\begin{align}
	JJ^{\mathsf{H}} = \begin{bmatrix}2& 2\frac{\Real(\alpha_j)}{\Imag(\alpha_j)}\\ 2\frac{\Real(\alpha_j)}{\Imag(\alpha_j)} & 4\frac{\Real(\alpha_j)^2}{\Imag(\alpha_j)^2} + 2\end{bmatrix} = LL^{\mathsf{H}}.
\end{align}
we can use
 $\breve{Z}=\sqrt{-2\Real(\alpha_j)}\begin{bmatrix}\Real(V_j)\ \Imag(V_j) \end{bmatrix}\left(L \otimes I_m\right)$ instead of $\hat{Z},$ as  $\breve{Z}\breve{Z}^{\mathsf{H}} =\hat{Z} \hat{Z}^{\mathsf{H}}.$

Altogether, to keep the iterates real, the columns
\begin{align}\label{eq:pad}
	\sqrt{-2\Real(\alpha_j)}\cdot\sqrt{2}\begin{bmatrix}\Real(V_j) + \frac{\Real(\alpha_j)}{\Imag(\alpha_j) }\Imag(V_j) & \sqrt{\frac{\Real(\alpha_j)^2}{\Imag(\alpha_j)^2}+1}\cdot\Imag(V_j) \end{bmatrix}
\end{align}
are added to the iterate $Z_{j-1}$, yielding the same approximation $Z_{j+1}$ as two steps of \Cref{alg:adi32} with shifts $\alpha_j$ and $\alpha_{j+1}=\overline{\alpha_j}$, but with real approximate Cholesky factors.

We refrain from stating the ADI iteration for solving the Sylvester equation \eqref{eq:sylv} as we will not make explicit use of it. Please see, e.g., \cite{Kue16} for a detailed description of a residual based variant or \cite{simoncini2016} for an outline of the development of the ADI iteration for Sylvester equations.

\section{Lyapunov equation}
\label{sec:gramian_as_ode}
In this section we will derive an ODE-based approximation of the solution $\mathcal{P}$ \eqref{eq:P} of the Lyapunov equation \eqref{eq:lyap}. The ODE will be solved via a Runge-Kutta method. The equivalence of our method to the ADI method \Cref{alg:adi32} for certain special Runge-Kutta methods will be discussed.

For reasons of simplicity we will not consider \eqref{eq:lyap} in full generality, only matrices $B = b\in \mathbb{R}^{n\times 1}$ with one column will be considered. The general case $B=[b_1, \dots, b_m]\in \mathbb{R}^{n\times m}$ can be reduced to
\begin{align}\label{eq:mimo_gramian_sum}
\begin{split}
	\mathcal{P} &= \int_0^\infty e^{At}BB^{\mathsf{T}}e^{A^{\mathsf{T}}t}\,\mathrm dt\\
	&= \sum_{i=1}^{m}\int_0^\infty e^{At}b_{i}b_{i}^{\mathsf{T}}e^{A^{\mathsf{T}}t}\,\mathrm dt.
\end{split}
\end{align}
Thus, our results extend to the case $m>1$ easily.

We will make use of the time dependent Gramian $P(t)$ which is given by
\begin{align}
	P(t) &\coloneqq \int_0^t e^{A\tau}bb^{\mathsf{T}}e^{A^{\mathsf{T}}\tau}\,\mathrm d\tau = \int_0^t h(\tau)h(\tau)^{\mathsf{T}}\,\mathrm d\tau 
\end{align}
where $h(t) \coloneqq e^{At}b.$
The functions $P(t)$ and $h(t)$ are the solutions of the system of ODEs
\begin{equation}
\begin{aligned}
	\frac{\mathrm d}{\mathrm dt}P(t) &= h(t)h(t)^{\mathsf{T}}, & P(0)=0,\\
	\frac{\mathrm d}{\mathrm dt}h(t) &= Ah(t), & h(0)=b.
\end{aligned}
	\label{eq:gramian_ode}
\end{equation}
Vectorizing \eqref{eq:gramian_ode} allows to write the system in the form \eqref{eq:ode}
\begin{align}\label{eq:myode}
	\frac{\mathrm d}{\mathrm dt}\begin{bmatrix}\vect(P(t))\\h(t)\end{bmatrix}=\begin{bmatrix}\vect(h(t)h(t)^{\mathsf{T}})\\Ah(t)\end{bmatrix} = f(t,y(t))
\end{align}
for the high dimensional solution function $y(t)  =[\vect(P(t))^\mathsf{T} ~~ h(t)^\mathsf{T}]^\mathsf{T}$, $y\colon \mathbb{R} \rightarrow \mathbb{R}^{n^2+n}$ which needs to be determined.
Clearly, $y(0) = [0^\mathsf{T}~b^\mathsf{T}]^\mathsf{T}.$

\subsection{Approximating $\mathcal{P}$ by Runge-Kutta methods}\label{sec31}
In order to solve \eqref{eq:myode}, we will make use of a Runge-Kutta  method with tableau \eqref{eq:butcher}.
For ease of notation, we will use $y(t) =[v(t)^\mathsf{T}~ h(t)^\mathsf{T}]^\mathsf{T},$
that is, $v(t) = \vect(P(t)).$ Moreover, the vector $k_i^{(j)}$ \eqref{eq:runge-kutta-k} is written in block form
\[
k_i^{(j)} = \begin{bmatrix}
\tilde{k}_i^{(j)} \\ \hat{k}_i^{(j)}\end{bmatrix}
\]
corresponding to the two blocks of $y(t) =[v(t)^\mathsf{T}~ h(t)^\mathsf{T}]^\mathsf{T}.$ 
We obtain
\begin{align}\label{eq:RK_vec}
\begin{split}
	v_j &= v_{j-1} + \omega_j \sum_{i=1}^s \beta_i \tilde{k}_i^{(j)}, \qquad v_0 = 0\in \mathbb{R}^{n^2},\\
	h_j &= h_{j-1} + \omega_j \sum_{i=1}^s \beta_i \hat{k}_i^{(j)}, \qquad  h_0 = b\in \mathbb{R}^{n},
\end{split}
\end{align}
with
\begin{align}
k_i^{(j)} &= \begin{bmatrix}
\tilde{k}_i^{(j)} \\ \hat{k}_i^{(j)}\end{bmatrix}
= f \left( t_{j-1}+\gamma_i\omega_j, \begin{bmatrix} 
v_{j-1} + \omega_j \sum_{\ell=1}^s \lambda_{i\ell} \tilde{k}_\ell^{(j)}\\
h_{j-1} + \omega_j \sum_{\ell=1}^s \lambda_{i\ell} \hat{k}_\ell^{(j)}
\end{bmatrix}\right) \nonumber\\
&= \begin{bmatrix}
\vect\left( (h_{j-1} + \omega_j \sum_{\ell=1}^s \lambda_{i\ell} \hat{k}_\ell^{(j)})
(h_{j-1} + \omega_j \sum_{\ell=1}^s \lambda_{i\ell} \hat{k}_\ell^{(j)})^\mathsf{H}    \right)\\
Ah_{j-1} + \omega_j \sum_{\ell=1}^s \lambda_{i\ell}A\hat{k}_\ell^{(j)} 
\end{bmatrix}. \label{eq:ohnektilde}
\end{align}
Please note the use of $\mathsf{H}$ instead of $\mathsf{T}$ in the expression above due to the
possibly complex valued $\lambda_{ij}$. Moreover, note that the expression in \eqref{eq:ohnektilde}
does not involve $\tilde{k}_i^{(j)}$.

\begin{rem}
Another idea to solve \eqref{eq:gramian_ode} is to employ two different $s$-stage Runge-Kutta methods, e.g.,
a Butcher tableau with $\Lambda$ and $\beta$ for $h(t)$ and a second Butcher tableau with $\hat{\Lambda}$ and 
$\hat{\beta}$ for $P(t).$ The only difference to \eqref{eq:RK_vec} and \eqref{eq:ohnektilde} as above is that in 
\eqref{eq:RK_vec} the equation for $v_j$ changes to $v_j =  v_{j-1} + \omega_j \sum_{i=1}^s \hat{\beta}_i \tilde{k}_i^{(j)}.$
The matrix $\hat{\Lambda}$ does not appear as the right hand side of $P'(t) = h(t)h(t)^\mathsf{T}$ does not depend on $P(t).$ We do not consider this any further here, as it would turn out in Section \ref{sec:invariant} that  $\hat{\beta}$ has to be chosen as $\hat{\beta}=\beta.$
\end{rem}

De-vectorizing the first equation of \eqref{eq:RK_vec} we obtain the iteration
\begin{align}\label{eq:RK_mat}
\begin{split}
P_j &= P_{j-1} + \omega_j \sum_{i=1}^s \beta_i \mathfrak{h}_i^{(j)} (\mathfrak{h}_i^{(j)})^\mathsf{H}, \qquad j = 1, \ldots, N\\
h_j &= h_{j-1} + \omega_j \sum_{i=1}^s \beta_i \hat{k}_i^{(j)}
\end{split}
\end{align}
with $P_0 = 0 \in \mathbb{R}^{n \times n},$  $h_0 = b \in \mathbb{R}^n$ and
\begin{align}
\mathfrak{h}_i^{(j)} &= h_{j-1} + \omega_j \sum_{\ell=1}^s \lambda_{i\ell} \hat{k}_\ell^{(j)},\\
\hat{k}_i^{(j)} &= Ah_{j-1} + \omega_j \sum_{\ell=1}^s \lambda_{i\ell}A\hat{k}_\ell^{(j)}
= A\mathfrak{h}_i^{(j)}
\end{align}
for $i =1, \ldots, s.$
Please note that due to $\Lambda \in \mathbb{C}^{s \times s}$ and $\beta\in \mathbb{C}^{s}$ all iterates $\mathfrak{h}_i^{(j)}$, $h_j$ as well as $\hat{k}_i^{(j)}$ will be complex valued vectors, while $P_j \in \mathbb{C}^{n \times n}$, $j = 1, \ldots, N.$

First all $\hat{k}_i^{(j)}$, $i = 1, \ldots, s$ need to be determined for a given $j.$ Then the corresponding iterates  $\mathfrak{h}_i^{(j)}$, $P_j$ and $h_j$ can be computed. Let  $K_{j}=[\hat{k}_1^{(j)},\ldots,\hat{k}_s^{(j)}]$.
Then we have
\begin{equation}
\label{def:KAH}
K_j = [ Ah_{j-1}, \ldots, Ah_{j-1}] + \omega_j AK_j\Lambda^\mathsf{T}.
\end{equation}
Via vectorization \eqref{def:KAH}  is reformulated as a linear system of equations of size $ns \times ns$
\begin{align}
	\label{eq:Kvec}
	\left(I_{ns} -\omega_j(\Lambda\otimes A)\right)\vect(K_{j}) &= \left[ (A{h}_{j-1})^{\mathsf{T}},\dots,(A{h}_{j-1})^{\mathsf{T}}\right]^{\mathsf{T}}\\
	\label{eq:Kveckron}
	&= (I_s\otimes A)(\mathds{1}_s\otimes I_n)h_{j-1}. 
\end{align}
Let $\mu_1, \dots, \mu_s$ and $\lambda_1, \dots, \lambda_n$ be the eigenvalues of $\Lambda$ and $A$ respectively. Then the
eigenvalues of $I_{ns} -\omega_j(\Lambda\otimes A)$ are given by
$1 - \omega_j \mu_p \lambda_q$, $p = 1, \ldots, s$, $q = 1, \ldots, n$.
Thus the solution of \eqref{eq:Kvec} is unique if and only if
\begin{align}\label{eq:eigcond}
	\mu_p\neq \frac{1}{\omega_j\lambda_q}
\end{align}
for all $p = 1, \ldots, s$ and $q = 1, \ldots, n$. 
Thus, in case we require $\Lambda$ to be chosen such that $\mu_p\in \mathbb{C}_{+},$ \eqref{eq:Kvec} has an unique solution.

An alternative to determining $K_j$ is given by the relation 
$K_j=A\mathcal{H}_j$ with $\mathcal{H}_j = [\mathfrak{h}_1^{(j)}, \ldots, \mathfrak{h}_s^{(j)}].$ 
This implies that
\begin{align}
	\label{def:KAHH}
	\mathcal{H}_{j} &=[h_{j-1}, \ldots, h_{j-1}] + \omega_j A \mathcal{H}_{j}\Lambda^{\mathsf{T}}
\end{align}
needs to be solved. One way to realize this is, e.g., by vectorizing the equation to transform it into a linear system of equations with the same coefficient matrix $I_{ns} -\omega_j(\Lambda\otimes A)$  as above. This system of equations can be solved uniquely if and only if the eigenvalues of $\Lambda$ and $A$ satisfy \eqref{eq:eigcond}. In particular, note that, as for \eqref{eq:Kvec},  we can not choose $\Lambda$ arbitrarily, at least \eqref{eq:eigcond} has to hold.

With the help of $K_j$ and $\mathcal{H}_j$ \eqref{eq:RK_mat} can be written as 
\begin{align}
	\label{eq:RK_matmat}
\begin{split}
P_j &= P_{j-1} + \omega_j \mathcal{H}_j \diag(\beta) \mathcal{H}_j^\mathsf{H}, \qquad j = 1, \ldots, N\\
h_j &= h_{j-1} + \omega_j  K_j\beta,
\end{split}
\end{align}
where $\beta \in \mathbb{C}^s$ is as in \eqref{eq:butcher}.
As
\[ 
P_j^\mathsf{H} = P_{j-1}^\mathsf{H} + \omega_j \mathcal{H}_j \diag(\overline{\beta}) \mathcal{H}_j^\mathsf{H},
\]
$P_j$ is only Hermitian in case $\beta \in \mathbb{R}^s.$ If in addition $\beta \in \mathbb{R}_+^s,$ then we can decompose
$\omega_j \diag(\beta)$ into $\diag(\omega_j\beta)^{\frac{1}{2}}\diag(\omega_j\beta)^{\frac{1}{2}}$ 
where $\diag(\omega_j \beta)^{\frac{1}{2}} \in \mathbb{R}^{s \times s}$ denotes the diagonal matrix with diagonal entries $\sqrt{\omega_j\beta_i}$,\, $i=1, \ldots, s$. This allows to express $\omega_j \mathcal{H}_j \diag(\beta) \mathcal{H}_j^\mathsf{H}$ as
\[
\omega_j \mathcal{H}_j \diag(\beta) \mathcal{H}_j^\mathsf{H}  = \left(\mathcal{H}_j \diag(\omega_j \beta)^{\frac{1}{2}}\right) \left(\mathcal{H}_j \diag(\omega_j \beta)^{\frac{1}{2}}\right) ^\mathsf{H}.
\]
Thus, for $\beta \in \mathbb{R}^s_+,$ $P_j$ is by construction a positive semi-definite matrix and can be expressed as
$P_j = Z_jZ_j^{\mathsf{H}}$ for some complex valued matrix $Z_j.$ Hence we have
\begin{align}\label{eq:ZZH}
Z_j Z_j^{\mathsf{H}} &= Z_{j-1}Z_{j-1}^{\mathsf{H}} + \omega_j \mathcal{H}_j \diag(\beta) \mathcal{H}_j^\mathsf{H}\\
&= \left[ Z_{j-1},  \mathcal{H}_j \diag(\omega_j \beta)^{\frac{1}{2}} \right]
\left[ Z_{j-1},  \mathcal{H}_j \diag(\omega_j \beta)^{\frac{1}{2}} \right]^\mathsf{H} \nonumber
\end{align}
where $\diag(\omega_j \beta)^{\frac{1}{2}} \in \mathbb{C}^{s \times s}$ denotes the diagonal matrix with diagonal entries $\sqrt{\omega_j\beta_i}$, $i=1, \ldots, s$. 

Instead of iterating on $P_j$ as in \eqref{eq:RK_matmat}, the above observation allows us in case $\beta\in \mathbb{R}^s_+$ to iterate on the low rank factor
\begin{align}
	Z_j &= [Z_{j-1}, \mathcal{H}_j\diag(\omega_j\beta)^{\frac{1}{2}}] \in \mathbb{C}^{n \times js}
\end{align}
which gains $s$ additional columns in every iteration step. The procedure to obtain the Gramian approximation described in this section is summarized in  \Cref{alg:quad}. We need to require that $\beta\in \mathbb{R}^s_+$ such that $P_j$ is positive semi-definite and that the eigenvalues of $\Lambda$ satisfy \eqref{eq:eigcond} in order to ensure that all linear system solves have a unique solution.
In case $N$ is such that $sN$ is less than $n,$ then $Z_j$ are low rank factors of $P_j$, $j = 1, \ldots, N$.

\begin{algorithm}[ht]
\caption{Low rank solution to \eqref{eq:lyap} via an $s$-stage Runge-Kutta method}
    \label{alg:quad}
    \begin{algorithmic}[1] 
	    \Input $A \in \mathbb{R}^{n \times n}$ stable, $b \in \mathbb{R}^n$, positive time step sizes $\left\{ \omega_1,\dots,\omega_{N} \right\}$ and a Butcher tableau (\ref{eq:butcher}) with $\Lambda \in \mathbb{C}^{s \times s}$ and $\beta \in \mathbb{R}^s_+$ which satisfies \eqref{eq:eigcond}
	    \Output  $Z\in \mathbb{C}^{n\times sN}$ with $Z Z^{\mathsf{H}}\approx \mathcal{P}$
    \State initialize $h_0=b$, $Z_0=[\ ]$
    \For {$j=1,\dots,N$}
    \State solve $K_j=[Ah_{j-1}, \ldots, Ah_{j-1}] + \omega_j AK_j\Lambda^{\mathsf{T}}$ for $K_j \in \mathbb{C}^{n \times s}$
    \State $\mathcal{H}=[h_{j-1}, \ldots, h_{j-1}] + \omega_j K_j\Lambda^{\mathsf{T}}$
    \State update 
$Z_j=[Z_{j-1}, \mathcal{H}\diag(\omega_j\beta)^{\frac{1}{2}}]$
    \State $h_{j} = h_{j-1} + \omega_jK_j\beta$
    \EndFor
    \State $Z=Z_N$
    \end{algorithmic}
\end{algorithm}

\subsection{Runge-Kutta methods which preserve an algebraic invariant}\label{sec:invariant}
Next we will take a closer look at  all possible Butcher tableaus in order to derive additional conditions for
suitable methods.

First, observe that the time dependent Gramian $P(t)$ will not exactly satisfy \eqref{eq:lyap}, a residual, the so-called Lyapunov residual,  will remain
\begin{equation}\label{eq:lyap_res}
\mathcal{L}(t) \coloneqq AP(t) + P(t)A^\mathsf{T} + bb^\mathsf{T}.
\end{equation}
Next, reconsider the derivative of $P(t)$ \eqref{eq:gramian_ode}
\begin{align}
	\frac{\mathrm d}{\mathrm dt}P(t) &= h(t)h(t)^{\mathsf{T}}\\
&=h(0)h(0)^{\mathsf{T}} + \int_0^t \frac{\mathrm d}{\mathrm d\tau} h(\tau)h(\tau)^{\mathsf{T}} \,\mathrm d\tau\\
	&= bb^{\mathsf{T}} + \int_0^t \left(Ah(\tau)h(\tau)^{\mathsf{T}}+h(\tau)h(\tau)^{\mathsf{T}}A^{\mathsf{T}}\right) \,\mathrm d\tau\\
	&= bb^{\mathsf{T}} + A\int_0^t h(\tau)h(\tau)^{\mathsf{T}}\,\mathrm d\tau +\int_0^t h(\tau)h(\tau)^{\mathsf{T}}\,\mathrm d\tau A^{\mathsf{T}}\\
	&= bb^{\mathsf{T}} + AP(t) +P(t)A^{\mathsf{T}}.
\end{align}

Thus, we have
	\begin{align}
		\mathcal{L}(t) = AP(t) + P(t)A^{\mathsf{T}} + bb^{\mathsf{T}} = h(t)h(t)^{\mathsf{T}},
		\label{eq:invariant}
	\end{align}
	for all $t$. Obviously, due to the right-hand side, the Lyapunov residual is of rank one.

	Our key idea is to use only those Butcher tableaus which guarantee that the iterates $P_j = Z_jZ_j^\mathsf{H}$ and $h_j$ satisfy the algebraic invariant \eqref{eq:invariant} in the sense
\begin{align}
	AP_j +P_jA^{\mathsf{T}} + bb^{\mathsf{T}} = h_jh_j^{\mathsf{H}}.
	\label{eq:invariant_complex}
\end{align}
Please note, that as $\Lambda \in \mathbb{C}^{s \times s}$ and $\beta\in \mathbb{C}^{s},$ our iterates $P_j$ and $h_j$
may be complex valued and thus, we need to modify $\mathsf{T}$ in \eqref{eq:invariant} to $\mathsf{H}$ in\eqref{eq:invariant_complex}.

Before we discuss which Butcher tableaus allow for \eqref{eq:invariant_complex} let us give an interpretation of \eqref{eq:invariant} and \eqref{eq:invariant_complex}.
Consider all tuples $(P,h),$ for which the invariant \eqref{eq:invariant_complex} is satisfied. They are located on the manifold
\begin{align}
	\mathcal{M} \coloneqq \left\{ (P,h) \text{ with }  P\in \mathbb{C}^{n\times n},\ h\in \mathbb{C}^{n\times 1} \mid AP+PA^{\mathsf{T}} + bb^{\mathsf{T}} - hh^{\mathsf{H}} = 0 \right\}.
	\label{def:lrr_manifold}
\end{align}
The solution $(P(t),h(t))$  of \eqref{eq:gramian_ode} lies on $\mathcal{M}$ for all times $t\in \mathbb{R}$ because of \eqref{eq:invariant}. As the Lyapunov residual \eqref{eq:invariant} is of rank one we will call $\mathcal{M}$ the {rank-one residual manifold}. The time dependent Gramian evolves on the rank-one residual manifold with $h(t)\to 0$ for $t\to \infty;$ see  Fig. \ref{fig:manifold1}. Enforcing \eqref{eq:invariant_complex} for the iterates $(P_j, h_j)$, in general we obtain $P_j$ which do not approximate the trajectory of the time dependent Gramian $P(t)$ but which are located on $\mathcal{M}.$ Therefore the approximation $P_N\approx \mathcal{P}$ is good when the iterate $h_N$ is small, because then the tuple $(P_N,h_N)\in \mathcal{M}$ is located close to $(\mathcal{P},0)\in \mathcal{M}$.

\begin{figure}
\def\svgwidth{.8\textwidth}
\centerline{ {\small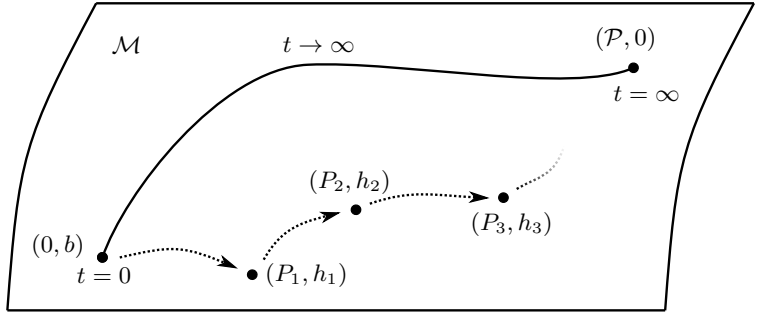}}
\caption{Solution of $\dot h=Ah,\, \dot P=hh^{\mathsf{T}}$ and iterates from \Cref{alg:quad} evolving on the rank-one residual manifold $\mathcal{M}$. }
	\label{fig:manifold1}
\end{figure}

\begin{thm}
	Let $A\in \mathbb{R}^{n \times n}$ be stable, $b \in \mathbb{R}^n$ and $\omega_i >0$, $i = 1, \ldots, j$. Consider a Butcher tableau \eqref{eq:butcher} with $\Lambda \in \mathbb{C}^{s \times s}$ and $\beta \in \mathbb{R}^s_+$ which satisfies \eqref{eq:eigcond}.
	After $j$ steps of \Cref{alg:quad}  the equation
	\begin{align}\label{eq:induction}
\begin{split}
		& AP_j + P_jA^{\mathsf{T}} + bb^{\mathsf{T}} \\
	&\qquad = h_jh_j^{\mathsf{H}} + \sum_{i=1}^{j}\omega_i^{2} K_{i}\left(\diag(\beta)\overline{\Lambda} + \Lambda^\mathsf{T}\diag(\beta)-\beta\beta^{\mathsf{T}}\right)K_{i}^{\mathsf{H}}
\end{split}
\end{align}
	holds for $P_j = Z_jZ_j^\mathsf{H}.$
	Thus for $K_i\neq 0$, $i = 1, \ldots, j$, the iterates $P_j$ and $h_j$ satisfy \eqref{eq:invariant_complex}  if and only if 
	\begin{align}
		\label{eq:famous_m}
		\diag(\beta)\overline{\Lambda} + \Lambda^{\mathsf{T}}\diag(\beta)-\beta\beta^{\mathsf{T}}=0
	\end{align}
	holds.
	\label{thm:invariant_approx}
\end{thm}
\begin{proof}
	For $j=0$ the statement is obviously true as $P_0 = Z_0Z_0^\mathsf{T}=0$ and $h_0=b$.

	For $j\in\mathbb{N}$ we first use \eqref{eq:ZZH} 
and then inductively \eqref{eq:induction} for $j-1$
\begin{equation}\begin{aligned}\label{eq:rek-inv}
		& AP_j + P_jA^{\mathsf{T}} + bb^{\mathsf{T}}\\
		&= AP_{j-1} + P_{j-1}A^{\mathsf{T}} +bb^{\mathsf{T}}
+ \omega_j A\mathcal{H}_j\diag(\beta)\mathcal{H}_j^{\mathsf{H}}
 + \omega_j \mathcal{H}_j\diag(\beta)\mathcal{H}_j^{\mathsf{H}}A^{\mathsf{T}}
\\
		&= h_{j-1}h_{j-1}^{\mathsf{H}} + \sum_{i=1}^{j-1}\omega_i^{2} K_{i}\left(\diag(\beta)\overline{\Lambda} + (\diag(\beta)\Lambda)^{\mathsf{T}}-\beta\beta^{\mathsf{T}}\right)K_{i}^{\mathsf{H}}\\
		&\qquad+ \omega_j A\mathcal{H}_j\diag(\beta)\mathcal{H}_j^{\mathsf{H}}
 + \omega_j \mathcal{H}_j\diag(\beta)\mathcal{H}_j^{\mathsf{H}}A^{\mathsf{T}}.
\end{aligned}\end{equation}
	Using $A\mathcal{H}_j=K_j$ and expanding $\mathcal{H}_j$ as in \eqref{def:KAHH} we obtain
	\begin{align}
		A\mathcal{H}_j\diag(\beta)\mathcal{H}_j^{\mathsf{H}} &= K_j\diag(\beta)\left([h_{j-1},\ldots, h_{j-1}]+\omega_j K_j\Lambda^{\mathsf{T}}\right)^{\mathsf{H}}\\
	&= K_j\beta h_{j-1}^{\mathsf{H}} + \omega_j K_j\diag(\beta)\overline{\Lambda} K_j^{\mathsf{H}}.
	\end{align}
	Inserting this in \eqref{eq:rek-inv} and adding a zero we find
	\begin{align}
		& AP_j + P_jA^{\mathsf{T}} + bb^{\mathsf{T}}\\
		&= h_{j-1}h_{j-1}^{\mathsf{H}} + \sum_{i=1}^{j-1}\omega_i^{2} K_{i}\left(\diag(\beta)\overline{\Lambda} + (\diag(\beta)\Lambda)^{\mathsf{T}}-\beta\beta^{\mathsf{T}}\right)K_{i}^{\mathsf{H}}\\
		&\qquad +\omega_j K_j\beta h_{j-1}^{\mathsf{H}} + \omega_j^{2}K_j\diag(\beta)\overline{\Lambda} K_j^{\mathsf{H}}
		+ \omega_j h_{j-1}\beta^{\mathsf{T}}K_j^{\mathsf{H}} + \omega_j^{2}K_j\Lambda^{\mathsf{T}}\diag(\beta)K_j^{\mathsf{H}}\\
		&\qquad+\omega_j^{2}K_j\beta\beta^{\mathsf{T}}K_j^{\mathsf{H}}-\omega_j^{2}K_j\beta\beta^{\mathsf{T}}K_j^{\mathsf{H}}\\
&= h_{j-1}h_{j-1}^{\mathsf{H}}+\omega_j K_j\beta h_{j-1}^{\mathsf{H}} + \omega_j h_{j-1}\beta^{\mathsf{T}}K_j^{\mathsf{H}} +\omega_j^{2}K_j\beta\beta^{\mathsf{T}}K_j^{\mathsf{H}}\\
		&\qquad +\omega_j^{2}K_j\left(\diag(\beta)\overline{\Lambda} + \Lambda^{\mathsf{T}}\diag(\beta)-\beta\beta^{\mathsf{T}}\right)K_j^{\mathsf{H}}\\
		&\qquad+\sum_{i=1}^{j-1}\omega_i^{2} K_{i}\left(\diag(\beta)\overline{\Lambda} + (\diag(\beta)\Lambda)^{\mathsf{T}}-\beta\beta^{\mathsf{T}}\right)K_{i}^{\mathsf{H}}\\
		&= h_jh_j^{\mathsf{H}} + \sum_{i=1}^{j}\omega_i^{2} K_{i}\left(\diag(\beta)\overline{\Lambda} + (\diag(\beta)\Lambda)^{\mathsf{T}}-\beta\beta^{\mathsf{T}}\right)K_{i}^{\mathsf{H}},
	\end{align}
as $h_j = h_{j-1}+\omega_j K_j \beta.$ This proves the first statement.
With $K_i\neq 0$ and $\omega_i > 0$ for $i = 1, \ldots, j,$  the second statement is immediate. This concludes the proof.
\end{proof}

There are numerous Butcher tableaus for which \eqref{eq:famous_m} holds.
First, observe that the diagonal entries of the equation $\diag(\beta)\overline{\Lambda} + (\diag(\beta)\Lambda)^{\mathsf{T}}-\beta\beta^{\mathsf{T}}=0$ imply that
\[ \beta_j \overline{\lambda_{jj}} + \beta_j \lambda_{jj} - \beta_j^2 = \beta_j \left(2\Real(\lambda_{jj})-\beta_j \right) = 0\]
for $j = 1, \ldots, s.$ Thus either $\beta_j = 0,$ or
\[
\beta_j = 2\Real(\lambda_{jj}), \qquad j = 1, \ldots, s.
\]
As $\beta_j \in \mathbb{R}_+$ is required, this implies that the diagonal elements of $\Lambda$ have to be in $\mathbb{C}_+$ whenever \eqref{eq:famous_m} is required.

The simplest $1$-stage Butcher tableaus satisfying \eqref{eq:famous_m} are given 
for an arbitrary (complex) number $\mu\in \mathbb{C}$ by
\begin{align}
	\Lambda=\mu,\ \beta=2\Real(\mu).
	\label{eq:mu_tab}
\end{align}	
The implicit midpoint rule
\[
\Lambda=\frac{1}{2},\ \beta=1
\]
is a prominent example of such a $1$-stage tableau. 
A simple  $2$-stage tableau which satisfies  \eqref{eq:famous_m} is the $2$-stage implicit Runge-Kutta method
\begin{align}\Lambda=\begin{bmatrix}
	\frac{1}{4} & \frac{1}{4}-\frac{1}{6}\sqrt{3}\\
	\frac{1}{4}+\frac{1}{6}\sqrt{3} & \frac{1}{4}
\end{bmatrix},\ \beta=\begin{bmatrix}
	\frac{1}{2}\\
	\frac{1}{2}
\end{bmatrix}.
\end{align}
The latter two methods belong to the family of Gauss-Legendre methods which are special $s$-stage implicit Runge-Kutta methods based on Gauss-Legendre quadrature. For $s\in \mathbb{N}$, the respective method is unique and satisfies \eqref{eq:famous_m}, see \cite[Lemma~5.3]{iserles_2008} and the subsequent corollary, where the matrix $M$ corresponds to the left-hand side of \eqref{eq:famous_m}.

Another family of methods for which \eqref{eq:famous_m} holds  is given by DIRK methods of the form
\begin{align}
	\Lambda = \begin{bmatrix}
		\mu_1& 0& 0& \cdots & \cdots & 0\\
		2\Real(\mu_1) & \mu_2& 0& \cdots& \cdots &0\\
2\Real(\mu_1)&2\Real(\mu_2)& \mu_3 &\ddots&  &\vdots \\
		\vdots & \vdots& \ddots&\ddots&\ddots &\vdots \\
\vdots&\vdots&&\ddots& \mu_{s-1} &0\\
		2\Real(\mu_1)&2\Real(\mu_2)&\cdots&\cdots& 2\Real(\mu_{s-1}) & \mu_s
	\end{bmatrix},\	\beta=\begin{bmatrix}
		2\Real(\mu_1)\\
		2\Real(\mu_2)\\
		\vdots\\
		2\Real(\mu_s)
	\end{bmatrix}
	\label{eq:cnf_tableau}
\end{align}
with $\mu_1, \dots, \mu_s \in \mathbb{C}_+$. In this case, it is easy to
verify whether the necessary and sufficient condition \eqref{eq:eigcond} 
\[
	\mu_p\neq \frac{1}{\omega_j\lambda_q}=\frac{1}{\omega_j |\lambda_q|^2}\overline{\lambda_q},
\]
is satisfied for all $p = 1, \ldots, s$ and $q = 1, \ldots, n$. As $\lambda_q$ is an eigenvalue of the stable matrix $A$, 
we have $\Real(\overline{\lambda_q}) <0$. 
Thus, for stable $A$ any DIRK method with a tableau \eqref{eq:cnf_tableau} satisfies \eqref{eq:eigcond}.

\subsection{Varying Butcher tableaus instead of one fixed Butcher tableau}\label{sec33}
Now we change the point of view on the Runge-Kutta methods. So far we have used the same tableau with $\Lambda$, $\beta$ in every iteration step. The time step sizes $\omega_j$ for $j=1, \ldots, N$ may vary. In the following we allow for varying tableaus with $\Lambda^{(j)} \in \mathbb{C}^{s \times s}$, $\beta^{(j)}\in \mathbb{C}^s$ during the iteration, in particular the matrices $\Lambda^{(j)}$ do not need to have the same eigenvalues. This implies the iteration 
	\begin{align}
		y_j &= y_{j-1} +  \omega_j\sum_{i=1}^{s}\beta^{(j)}_ik_i^{(j)}, & j &= 1, \ldots, N,\\
		k_i^{(j)} &= f\Big(t_{j-1}+\gamma^{(j)}_i,\, y_{j-1}+\omega_j\sum_{\ell=1}^{s}\lambda^{(j)}_{i\ell}k_\ell^{(j)}\Big),\ & i&=1,\dots,s,
	\end{align}
instead of \eqref{eq:runge-kutta} and \eqref{eq:runge-kutta-k}.

Clearly, we would like to choose $\Lambda^{(j)}$, $\beta^{(j)}$ such that \eqref{eq:famous_m} is satisfied. In that case, $\omega_j\Lambda^{(j)},$ $\omega_j\beta^{(j)}$ also satisfy \eqref{eq:famous_m}. Thus, there is no need for choosing a different time step size in every iteration step, this is in a sense already dealt with by allowing different  $\Lambda^{(j)}$, $\beta^{(j)}$ in every iteration step. Therefore, we set $\omega_j=1$ for $j=1, \dots, N,$ whenever we allow for different  $\Lambda^{(j)}$, $\beta^{(j)}$ in every iteration step,
\begin{align}
		\label{eq:runge-kutta_ohne}
		y_j &= y_{j-1} + \sum_{i=1}^{s}\beta^{(j)}_ik_i^{(j)}, & j &= 1, \ldots, N,\\
		\label{eq:runge-kutta-k_ohne}
		k_i^{(j)} &= f\Big(t_{j-1}+\gamma^{(j)}_i,\, y_{j-1}+\sum_{\ell=1}^{s}\lambda^{(j)}_{i\ell}k_\ell^{(j)}\Big),\ & i&=1,\dots,s.
	\end{align}
	\Cref{alg:quad} has to be modified accordingly. In step 3 and 4 the term $\omega_j\Lambda$ has to be replaced by $\Lambda^{(j)},$ while in step 5 and 6 the term  $\omega_j\beta$ has to be replaced by $\beta^{(j)}.$ Apparently \Cref{thm:invariant_approx} remains true even when different tableaus are used in every step as long as \eqref{eq:famous_m} holds for $\Lambda^{(i)}$ and $\beta^{(i)}$, $i=1, \dots, j.$ In its proof the tableaus $\Lambda$, $\beta$ have to be replaced as described above for \Cref{alg:quad}.

\subsection{$1$-stage Runge-Kutta methods}\label{sec34}
Let us consider the case $s=1$ and the Butcher tableau with $\Lambda^{(j)}=\mu_j \in \mathbb{C}_{+}$ and $\beta^{(j)}=2\Real(\mu_j) \in \mathbb{R}_{+}$ in iteration step $j.$ 
With the condition $\Real({\mu_j})>0$ we make sure \eqref{eq:eigcond} is satisfied and all linear systems solves in \Cref{alg:quad} will have a unique solution.

For $s = 1,$ in step 3 of \Cref{alg:quad} the $n\times n$ linear system of equations
\begin{align}
    \left( I-\mu_j A \right)K_j = Ah_{j-1}
\end{align}
has to be solved to obtain $K_j$. This approach is summarized in \Cref{alg:quad_adi}.

\begin{algorithm}[!ht]
\caption{Low rank solution to \eqref{eq:lyap} via $1$-stage Runge-Kutta methods}
    \label{alg:quad_adi}
    \begin{algorithmic}[1] 
	    \Input $A\in \mathbb{R}^{n\times n}$ stable, $b\in \mathbb{R}^{n}$, parameters $\left\{ \mu_1,\dots,\mu_{N} \right\}\subset \mathbb{C}_{+}$
	    \Output  factor $Z\in \mathbb{C}^{n\times N}$ with $Z Z^{\mathsf{H}}\approx \mathcal{P}$
	    \State initialize $h_0=b$, $Z_0=[\ ]$
    \For {$j=1,\dots,N$}
    \State solve $(I-\mu_j A)K_j=Ah_{j-1}$ for $K_j \in \mathbb{C}^{n \times 1}$
    \State $\mathcal{H}_j=h_{j-1} + \mu_j K_j$
    \State update  $Z_j=[Z_{j-1}, \sqrt{2\Real(\mu_j)}\mathcal{H}_j]$
    \State $h_{j} = h_{j-1} + 2\Real(\mu_j)K_j$
    \EndFor
    \State $Z=Z_N$
    \end{algorithmic}
\end{algorithm}

\Cref{alg:quad_adi}  is equivalent to \Cref{alg:adi32} (with $m=1$) as shown in the following theorem. 
\begin{thm}
	\label{thm:algo_equiv}
Let $A \in  \mathbb{R}^{n\times n}$  be stable and $B=b\in \mathbb{R}^{n \times 1}$ (that is, $m=1$ in \Cref{alg:adi32}).
Let the parameters  in \Cref{alg:quad_adi}  and \Cref{alg:adi32} be chosen such that $\alpha_j=-\mu_j^{-1} \in \mathbb{C}_-$ for $j=1, \dots, N$.
Let $W_0=b$ and $W_j,$ $V_j,$ $j = 1, \ldots, N$ be determined by  \Cref{alg:adi32}. Let $h_0=b$ and $h_j,$ $\mathcal{H}_j,$ $j = 1, \ldots, N$ be determined by \Cref{alg:quad_adi}.
Then $W_j= c_jh_j$ and $\sqrt{-2\Real(\alpha_j)}V_j=d_j\sqrt{2\Real(\mu_j)}\mathcal{H}_j$ holds for some constants $c_j, d_j\in \mathbb{C}$ with $|c_j|=|d_j|=1$. Thus the approximation $Z_j Z_j^{\mathsf{H}} \approx \mathcal{P}$ to the Gramian is the same in every step of both algorithms.
\end{thm}

\begin{proof}
	The initialization of the algorithms gives us $W_0=b=h_0$ and thus $c_0=1$.

	From \eqref{def:KAHH} with $s=1$ and $\omega_j\Lambda = \mu_j$ we find $\mathcal{H}_j = h_{j-1} + \mu_j A\mathcal{H}_j$ and thus
	\begin{align}
		\mathcal{H}_j = (I_n-\mu_jA)^{-1}h_{j-1}.
	\end{align}
	Making use of step 4 in \Cref{alg:quad_adi} we obtain
	\begin{align}\label{eq:ijk}
		(I_n-\mu_jA)^{-1}h_{j-1} = \mathcal{H}_j = h_{j-1} + \mu_jK_j.
	\end{align}

	Via induction we find from $W_{j-1}=c_{j-1}h_{h-1}$ for $W_j$ as in \Cref{alg:adi32}
	\begin{align}
		W_j &= W_{j-1} - 2\Real(\alpha_j)V_j\\
		&= W_{j-1} - 2\Real(\alpha_j)(A+\alpha_j I_n)^{-1}W_{j-1}\\
		&= c_{j-1}h_{j-1} - 2\Real(-\mu_j^{-1})(A-\mu_j^{-1}I_n)^{-1}c_{j-1}h_{j-1}\\
		&= c_{j-1}\left(h_{j-1} + 2\Real(\mu_j^{-1})(-\mu_j)(I_n - \mu_jA)^{-1}h_{j-1}\right)\\
		&\stackrel{\eqref{eq:ijk}}{=} c_{j-1}\left(h_{j-1} - 2\Real(\mu_j^{-1})\mu_j(h_{j-1}+\mu_jK_j)\right)\\
		&= c_{j-1}\left( (1-2\Real(\mu_j^{-1})\mu_j)h_{j-1} - 2\Real(\mu_j^{-1})\mu_j^2K_j)\right)\\
		&= -c_{j-1}\frac{\mu_j}{\overline\mu_j}\left( h_{j-1} + 2\Real(\mu_j)K_j\right)\\
		&= c_{j}h_j
	\end{align}
	with $1-2\Real(\mu_j^{-1})\mu_j=-\frac{\mu_j}{\overline\mu_j}$, $-2\Real(\mu_j^{-1})\mu_j^2=-\frac{\mu_j}{\overline\mu_j}2\Real(\mu_j)$ and $c_j\coloneqq-c_{j-1}\frac{\mu_j}{\overline\mu_j}$.

	Further we find for $V_j$ as in \Cref{alg:adi32}
	\begin{align}
		\sqrt{-2\Real(\alpha_j)}V_j &= \sqrt{-2\Real(\alpha_j)}(A+\alpha_jI_n)^{-1}W_{j-1}\\
		&= \sqrt{2\Real(\mu_j^{-1})}\mu_j(\mu_j A-I_n)^{-1}W_{j-1}\\
		&= -\sqrt{2\Real(\mu_j)|\mu_j|^{-2}}\mu_j(I_n-\mu_j A)^{-1}c_{j-1}h_{j-1}\\
		&\stackrel{\eqref{eq:ijk}}{=} -\sqrt{2\Real(\mu_j)}\frac{\mu_j}{|\mu_j|}c_{j-1}\mathcal{H}_j\\
		&= d_j\sqrt{2\Real(\mu_j)}\mathcal{H}_j
	\end{align}
	with $d_j\coloneqq -\frac{\mu_j}{|\mu_j|}c_{j-1}$. The observation $c_j\overline{c_j}=d_j\overline{d_j}=1$ concludes the proof.
\end{proof}

As \Cref{alg:quad_adi} is equivalent to \Cref{alg:adi32}, all properties of \Cref{alg:adi32} hold for \Cref{alg:quad_adi}
as well. In particular, complex valued iterates can be avoided if the set of parameters $\{\mu_j\}_{j=1}^N$ is proper. See \Cref{sec:adi}, \cite{penzl99},\cite[Chp. 4.3]{li2002}, \cite{Benner2012} or \cite[Chp. 4.1.4]{Kue16} for more.

In order to derive \Cref{alg:quad_adi} we have used different $1\times 1$-tableaus in each iteration step. In general,
larger tableaus allow for more accurate quadrature rules. Hence, their use might lead to an approximation of $\mathcal{P}$ which is more accurate than the one obtained by \Cref{alg:quad_adi}. However, as we show next, an algorithm using  larger DIRK tableaus 
can be reduced to \Cref{alg:quad_adi}.
In particular, $s$ steps with different $1\times 1$-tableaus $\Lambda^{(i)} = \mu_i$ and  $\beta^{(i)} = 2\Real(\mu_i)$, $i = 1, \ldots, s$ (as in Algorithm \ref{alg:quad_adi}, $N = s$) are equivalent to  one step with a particular $s$-stage DIRK tableau with $\Lambda$ and $\beta$ (as in Algorithm \ref{alg:quad}).
For $s$ steps of Algorithm \ref{alg:quad_adi} we have from step 6 
\begin{align}
	h_s &= h_{s-1} + 2\Real(\mu_s) K_{s}
	= h_0 + \sum_{i=1}^{s}2\Real(\mu_i)K_{i}\\
	&= h_0 + [K_1, \dots, K_s]\begin{bmatrix}2\Real(\mu_1)\\ \vdots \\ 2\Real(\mu_s)\end{bmatrix}\label{eq:beta}
\end{align}
with 
\begin{align}
	\label{eq:K}
	K_{i} &= Ah_{i-1}+\mu_iAK_{i}
	= Ah_0 + \sum_{j=1}^{i-1}2\Real(\mu_j)AK_{j} + \mu_iAK_{i}\in \mathbb{C}^{n\times 1}
\end{align}
for $i=1,\dots, s$. Merging \eqref{eq:K} for $i=1, \dots, s$ into one equation yields
\begin{align}
	\label{eq:KK}
	[K_1, \dots, K_s] &= [Ah_0,\dots, Ah_0] + [AK_1,\dots, AK_s]\begin{bmatrix}
		\mu_1& 2\Real(\mu_1) & \cdots & 2\Real(\mu_1)\\
		0 & \mu_2& \cdots& 2\Real(\mu_2)\\
		\vdots & \vdots& \ddots&\vdots\\
		0&0&\cdots&\mu_s
	\end{bmatrix}.
\end{align}
Hence \eqref{eq:beta} and \eqref{eq:K} are identical to step 6 and step 3 of \Cref{alg:quad} with $\omega_1=1$, where $[K_1, \dots, K_s]\in \mathbb{C}^{n\times s}$ from \eqref{eq:KK} corresponds to the $n\times s$ matrix $K_1$ in step 3 of \Cref{alg:quad}. Thus \Cref{alg:quad_adi} is equivalent to one step of \Cref{alg:quad} with $s=N$, time step size $\omega_1=1$ and the Butcher tableau \eqref{eq:cnf_tableau}.

In summary, Algorithm \ref{alg:quad_adi} is equivalent to the ADI method (\Cref{alg:adi32}) as well as
equivalent to one step of \Cref{alg:quad} for a DIRK method with $s=N.$
\begin{thm}
Let $A \in \mathbb{R}^{n \times n}$ be stable and $b\in \mathbb{R}^n.$
Let $\mu_j \in \mathbb{C}_+,\, j = 1, \ldots, s$ be given. Let $\Lambda \in \mathbb{C}^{s \times s}$ and $\beta\in \mathbb{R}_+^s$ be as in \eqref{eq:cnf_tableau}. One step of \Cref{alg:quad}  with $\Lambda,\, \beta$ is equivalent to $s$ steps
of Algorithm \ref{alg:quad_adi} with $N=s$ and equivalent to the ADI method (\Cref{alg:adi32}) with $\alpha_j = -\mu_j^{-1}.$
\end{thm}

Please note that due to the observation \eqref{eq:mimo_gramian_sum}  we have considered only a vector $b \in \mathbb{R}^{n}$ in our discussion and, in particular, in \Cref{alg:quad_adi}. However, \Cref{alg:quad_adi} can easily be adapted to a problem with $B\in \mathbb{R}^{n\times m}.$  The vector $b$ in \Cref{alg:quad_adi} just has to be replaced by  a matrix $B\in \mathbb{R}^{n\times m}$.

\subsection{A multiplicative update formula for $h_j$}\label{subsec_multupdate}
We now focus on the iterate $h_j = h_{j-1} + \omega_j K_j\beta$ as computed in step 6 of Algorithm \ref{alg:quad} from a Butcher tableau with $\Lambda \in \mathbb{C}^{s \times s}$ and $\beta\in \mathbb{C}^s,$
where $K_j$ is as in \eqref{def:KAH}.
As before, we will assume that $\omega_j = 1$ by moving the time step size into the Butcher tableau such that instead of one $\Lambda$ and $\beta$ we are now using $N$ different $\Lambda^{(j)} \in \mathbb{C}^{s \times s}$ and $\beta^{(j)}\in \mathbb{C}^s$. Thus we consider
$h_j = h_{j-1}+K_j\beta^{(j)}.$
Our goal is to rewrite the update rule as a multiplicative one; $h_j = M_j h_{j-1}.$

With \eqref{eq:Kveckron} we find
\begin{align}
	K_j\beta^{(j)} &= ((\beta^{(j)})^{\mathsf{T}}\otimes I_n)\vect(K_j)\\
	&= ((\beta^{(j)})^{\mathsf{T}}\otimes I_n)(I_{ns}-\Lambda^{(j)}\otimes A)^{-1}(I_s\otimes A)(\mathds{1}_s\otimes I_n)h_{j-1}\\
	&= ((\beta^{(j)})^{\mathsf{T}}\otimes I_n\left\{ (I_s^{-1}\otimes A^{-1})(I_{ns}-\Lambda^{(j)}\otimes A) \right\}^{-1}(\mathds{1}_s\otimes I_n)h_{j-1}\\
	&= ((\beta^{(j)})^{\mathsf{T}}\otimes I_n)\left\{ (I_s^{-1}\otimes A^{-1})-\Lambda^{(j)}\otimes I_n \right\}^{-1}(\mathds{1}_s\otimes I_n)h_{j-1}\\
	&= ((\beta^{(j)})^{\mathsf{T}}\otimes I_n)\left\{ (I_{ns}-\Lambda^{(j)}\otimes A)(I_s^{-1}\otimes A^{-1}) \right\}^{-1}(\mathds{1}_s\otimes I_n)h_{j-1}\\
	&= ((\beta^{(j)})^{\mathsf{T}}\otimes I_n)(I_s\otimes A)( I_{ns}-\Lambda^{(j)}\otimes A )^{-1}(\mathds{1}_s\otimes I_n)h_{j-1}\\
	&= ((\beta^{(j)})^{\mathsf{T}}\otimes A)(I_{ns}-\Lambda^{(j)}\otimes A)^{-1}(\mathds{1}_s\otimes I_n)h_{j-1}.
\end{align}

With \eqref{eq:kronswap} we have $\Lambda^{(j)}\otimes A = Q(A \otimes \Lambda^{(j)})Q^{\mathsf{T}}$ 
for a perfect shuffle permutation matrix $Q \in \mathbb{R}^{ns \times ns}.$ Moreover, as $Q$ is orthogonal, it holds that
$I_n \otimes I_s = I _{ns} = QQ^\mathsf{T} = Q(I_s \otimes I_n)Q^{\mathsf{T}}.$ Hence,
\begin{align}
	K_j\beta^{(j)} &= ((\beta^{(j)})^{\mathsf{T}}\otimes A)\left\{ Q( I_n\otimes I_s - A \otimes \Lambda^{(j)})Q^{\mathsf{T}} \right\}^{-1}(\mathds{1}_s\otimes I_n)h_{j-1}\\
	&= {I_{n}}((\beta^{(j)})^{\mathsf{T}}\otimes A)Q\left\{ I_n\otimes I_s-A \otimes \Lambda^{(j)} \right\}^{-1} Q^{\mathsf{T}} (\mathds{1}_s\otimes I_n){I_n} h_{j-1}\\
	&= (A\otimes (\beta^{(j)})^{\mathsf{T}})\left( I_{ns}-A \otimes \Lambda^{(j)} \right)^{-1}(I_n\otimes \mathds{1}_s)h_{j-1}
\end{align}
as the perfect shuffle matrix for a vector is the identity.

Thus the iterate $h_j$ is obtained from $h_{j-1}$ via
\begin{align}
	h_j = M_j h_{j-1}
	\label{eq:hMh}
\end{align}
with the iteration matrix
\begin{align}
	\label{eq:Miter}
	M_j &\coloneqq M_j(A) = I_n + (A\otimes (\beta^{(j)})^{\mathsf{T}})\left( I_{ns}-A \otimes \Lambda^{(j)} \right)^{-1}(I_n\otimes \mathds{1}_s).
\end{align}
The matrix valued function $M(z) = I_n + (z\otimes \beta^{\mathsf{T}})\left( I_{ns}-z \otimes \Lambda \right)^{-1}(I_n\otimes \mathds{1}_s)$ can be viewed as a generalization of the stability function $R(z)=1+z\beta^{\mathsf{T}}(I-z\Lambda)^{-1}\mathds{1}_s$ of the corresponding Runge-Kutta method.

For a diagonalizable system matrix $A=VDV^{-1}$ with the matrix of right eigenvectors $V$ and the diagonal matrix $D=\diag(\lambda_1, \dots, \lambda_n)$ containing the eigenvalues of $A$ on the diagonal, the iteration matrix $M_j$ from \eqref{eq:Miter} simplifies considerably. We define the stability function $R_{(j)}(z)=1+z(\beta^{(j)})^{\mathsf{T}}(I-z\Lambda^{(j)})^{-1}\mathds{1}_s$ for a Runge-Kutta method with $\Lambda^{(j)}$ and $\beta^{(j)}$ and see
\begin{align}
	M_j &= I_n + (A\otimes (\beta^{(j)})^{\mathsf{T}})\left( I_{ns}-(A \otimes \Lambda^{(j)}) \right)^{-1}(I_n\otimes \mathds{1}_s) \nonumber\\
	&= I_n + V(D\otimes (\beta^{(j)})^{\mathsf{T}})(V^{-1}\otimes I_s) \nonumber\\
	&\qquad \cdot\left\{ I_n\otimes I_s-(V\otimes I_s)(D \otimes \Lambda^{(j)})(V^{-1}\otimes I_s) \right\}^{-1}(I_n\otimes \mathds{1}_s) \nonumber\\
	&= I_n + V(D\otimes (\beta^{(j)})^{\mathsf{T}})\left( I_n\otimes I_s-(D \otimes \Lambda^{(j)})\right)^{-1}(I_n\otimes \mathds{1}_s)V^{-1} \nonumber\\
	&= V\left(I_n + \left[\begin{smallmatrix}\lambda_1(\beta^{(j)})^{\mathsf{T}}& &\\&\ddots &\\ &&\lambda_n(\beta^{(j)})^{\mathsf{T}} \end{smallmatrix}\right]
	 \left[\begin{smallmatrix}I_s-\lambda_1\Lambda^{(j)}& &\\&\ddots &\\ &&I_s-\lambda_n\Lambda^{(j)} \end{smallmatrix}\right]^{-1}\left[\begin{smallmatrix}\mathds{1}_s& &\\&\ddots &\\ &&\mathds{1}_s \end{smallmatrix}\right]\right)V^{-1} \nonumber \\
	&= V\left[\begin{smallmatrix}1+\lambda_1(\beta^{(j)})^{\mathsf{T}}(I_s-\lambda_1\Lambda^{(j)})^{-1}\mathds{1}_s&&\\&\ddots &\\ &&1+\lambda_n(\beta^{(j)})^{\mathsf{T}}(I_s-\lambda_n\Lambda^{(j)})^{-1}\mathds{1}_s\end{smallmatrix}\right]V^{-1} \nonumber\\
	&= V\begin{bmatrix}R_{(j)}(\lambda_1)& &\\&\ddots &\\ &&R_{(j)}(\lambda_n) \end{bmatrix}V^{-1},
	\label{eq:iter_mat_diag}
\end{align}
i.e. the iteration matrix is determined by the stability function $R_{(j)}(z)$ corresponding to the Butcher tableau with $\Lambda^{(j)}$, $\beta^{(j)}$ as well as by the eigenvalues and eigenvectors of the system matrix $A$. For non-diagonalizable system matrices we have no explicit formula in terms of the stability function. However, $M_j$---a composition of continuous functions---depends continuously on $A$ and the diagonalizable matrices are dense in the set of all matrices. Thus $M_j$ is (implicitly) determined by the stability function of the utilized Runge-Kutta method for non-diagonalizable matrices $A$, too.

Tableaus with the same stability function yield the same approximation $h_j$ as $h_j = M_j h_{j-1}$ holds
with $M_j$ as in \eqref{eq:iter_mat_diag} for each of the different methods.
In case the tableaus satisfy \eqref{eq:invariant_complex}, then
\begin{align}
	AZ_jZ_j^{\mathsf{H}} + Z_jZ_j^{\mathsf{H}}A^{\mathsf{T}} + (bb^{\mathsf{T}}-h_jh_j^{\mathsf{H}})=0
	\label{eq:P_det_by_h}
\end{align}
holds, and $Z_jZ_j^{\mathsf{H}}$ (but not $Z_j$) is determined by $h_j$ via \eqref{eq:P_det_by_h}. Thus, all tableaus
with the same stability function yield the same approximation $P_j = Z_jZ_j^{\mathsf{H}}$ (as long as the tableaus satisfy \eqref{eq:eigcond}, \eqref{eq:famous_m} and $\beta\in\mathbb{R}^s_+$).

\subsection{Runge-Kutta methods satisfying \eqref{eq:famous_m}} \label{sec36}
We now investigate when two $s$-stage Runge-Kutta methods have the same stability function. Assume that $\Lambda$ and $\beta$ are as in \Cref{thm:invariant_approx} and satisfy \eqref{eq:famous_m}, $0=\diag(\beta)\overline{\Lambda} + \Lambda^{\mathsf{T}}\diag(\beta)-\beta\beta^{\mathsf{T}}.$ Then
\[
	-\overline{\Lambda} =\diag(\beta)^{-1}(\Lambda^{\mathsf{T}} - \beta\mathds{1}^{\mathsf{T}}_s)\diag(\beta)
\]
as  $\beta_j > 0,$ $j = 1, \ldots, s$, proving that the matrices $-\overline{\Lambda}$ and $\Lambda - \mathds{1}_s\beta^{\mathsf{T}}$ are similar, i.e. 
\begin{align}\label{eq:sim_tabs}
-\overline{\Lambda} \quad \sim \quad \Lambda - \mathds{1}_s\beta^{\mathsf{T}}.
\end{align}
Let $\Lambda$ have eigenvalues $\mu_1, \dots, \mu_s$. Then with the determinant based characterization of the stability function \eqref{eq:stab} we find that
	\begin{align}
		R(z)&=\frac{\det(I+z(\mathds{1}_s\beta^{\mathsf{T}}-\Lambda))}{\det(I-z\Lambda)} 
= \frac{\det(I+z\overline{\Lambda}))}{\det(I-z\Lambda)} \nonumber \\
		\label{eq:char_poly_zinv}
		&= \frac{(1+\overline{\mu_1}z)\cdots(1+\overline{\mu_s}z)}{(1-\mu_1z)\cdots(1-\mu_sz)}
	\end{align}
	holds. This is just the stability function of a DIRK method as given in  \eqref{eq:cnf_tableau}. Thus  any
method based on a Butcher tableau with $\Lambda,\, \beta$ such that $\Lambda$ has the eigenvalues $\mu_1, \dots, \mu_s$ and $\beta\in\mathbb{R}^s_+$ satisfies \eqref{eq:famous_m} is equivalent to a DIRK method \eqref{eq:cnf_tableau}. Note that obviously the order of the parameters $\mu_i$ is irrelevant. We summarize our findings in the following theorem.
\begin{thm}\label{thm:dim1_tableaus}
Let $A \in \mathbb{R}^{n \times n}$ be stable and $b\in \mathbb{R}^n.$
Let a Butcher tableau \eqref{eq:butcher} with $\Lambda \in \mathbb{C}^{s \times s}$ and $\beta \in \mathbb{R}^s_+$ satisfying \eqref{eq:famous_m} be given. 
Let $\Lambda$ have eigenvalues $\mu_1, \ldots, \mu_s\in \mathbb{C}_{+}$, so \eqref{eq:eigcond} is guaranteed. Moreover, \eqref{eq:sim_tabs} holds. Then the method based on this tableau is equivalent to the DIRK method given in \eqref{eq:cnf_tableau} and therefore equivalent to $s$ steps with the $1$-stage tableaus $\Lambda^{(j)}=\mu_j$ and $\beta^{(j)}=2\Real(\mu_j)$ for $j=1,\ldots,s$, i.e., equivalent to $s$ steps of \Cref{alg:quad_adi}.
\end{thm}

Let two different Butcher tableaus with $\Lambda\in \mathbb{C}^{s \times s},\, \beta\in\mathbb{R}^s_+$ and $\widetilde{\Lambda}\in \mathbb{C}^{s \times s},\, \widetilde{\beta}\in\mathbb{R}^s_+$ be given such that $\Lambda$ and $\widetilde{\Lambda}$ have the same eigenvalues, $\sigma(\Lambda) = \sigma(\widetilde{\Lambda}).$ Further assume that $\Lambda,\, \beta$ are chosen such that \eqref{eq:famous_m}, \eqref{eq:eigcond} and \eqref{eq:sim_tabs} hold.
Please note that Theorem \ref{thm:dim1_tableaus} does not imply that the method based on the Butcher tableau with $\widetilde{\Lambda},\, \widetilde{\beta}$ is equivalent to the one based on $\Lambda,\, \beta.$ Only in case $\widetilde{\Lambda},\, \widetilde{\beta}$ also satisfies \eqref{eq:famous_m}, \eqref{eq:eigcond} and \eqref{eq:sim_tabs}, the two methods are equivalent. In the following we demonstrate this statement with particular tableaus.

Consider $2$ steps with the method based on the $1$-stage tableaus $\Lambda^{(1)}=\mu,\, \beta^{(1)}=2\Real(\mu)$ and $\Lambda^{(2)}=\overline{\mu},\, \beta^{(2)}=2\Real(\mu)$ where $\mu\in\mathbb{C}$ is chosen such that $\Real(\mu)\neq 0$ and \eqref{eq:eigcond} holds. This is equivalent to one step of the method based on the DIRK method with
\[
\Lambda = \begin{bmatrix} \mu & 0\\ 2\Real(\mu) & \overline{\mu}\end{bmatrix} \in \mathbb{C}^{2 \times 2},\, \beta = \begin{bmatrix} 2\Real(\mu) \\
2\Real(\mu)\end{bmatrix} \in \mathbb{R}^2.
\]
The matrix $\Lambda$ is similar to 
\[
\widetilde{\Lambda} = \begin{bmatrix} \overline{\mu} & 0\\ 0 & \mu\end{bmatrix} =
V^{-1}\Lambda V\in \mathbb{C}^{2\times 2}
\]
as well as to
\begin{equation}
\widehat{\Lambda} = \begin{bmatrix}\Real(\mu) & -\Imag(\mu)\\ \Imag(\mu)& \Real(\mu)\end{bmatrix}\in \mathbb{R}^{2\times 2}
\label{eq:lambdahat}
\end{equation}
as $\widetilde{\Lambda} = V^{-1}\Lambda V$ and $\widehat{\Lambda} = W^{-1}\Lambda W$ with
$V = \left[\begin{smallmatrix} 0 & 1 \\ \imath z & -\imath z\end{smallmatrix}\right],$
$W = \frac{1}{\sqrt{2}}\left[\begin{smallmatrix} 1 & \imath\\ 0 &2z\end{smallmatrix}\right]$ and $z = \frac{\Real(\mu)}{\Imag(\mu)}.$
A quick check reveals that for $\delta\in\mathbb{C}^2$
\[0=\diag(\delta)\overline{\widetilde{\Lambda}} + \widetilde{\Lambda}^{\mathsf{T}}\diag(\delta)-\delta\delta^{\mathsf{T}}
\]
as well as
\[0=\diag(\delta)\overline{\widehat{\Lambda}} + \widehat{\Lambda}^{\mathsf{T}}\diag(\delta)-\delta\delta^{\mathsf{T}}
\]
is satisfied only if either $\delta_1 =0$ or $\delta_2 =0.$ Thus no method based on a Butcher tableau with $\widetilde{\Lambda}$ or $\widehat{\Lambda}$ is equivalent to the method based on the Butcher tableau with $\Lambda,\, \beta.$ 

Another matrix similar to $\Lambda$ is given by
\[\label{eq:breveL}
	\breve{\Lambda}=\begin{bmatrix}
		\Real(\mu)&\Real(\mu)+\varphi |\mu|\\
		\Real(\mu) -\varphi |\mu|&\Real(\mu)
		\end{bmatrix}\in \mathbb{R}^{2\times 2}
\]
with $\varphi=\sgn(\Imag(\mu))=\frac{\Imag(\mu)}{|\Imag(\mu)|}.$ $\breve{\Lambda}$ is similar to $\widehat{\Lambda}$ as $\widehat{\Lambda} = S^{-1}\breve{\Lambda}S$ with
\begin{align}\label{eq:S}
	S &= \underbrace{\frac{1}{\sqrt{2}}\begin{bmatrix}1 & 1\\ 1 &-1\end{bmatrix}}_{\eqqcolon Q} 
\underbrace{\sqrt{2} \begin{bmatrix}1&z \\ 0& \sqrt{z^2+1}\end{bmatrix}}_{= L^\mathsf{T}}
\end{align}
and thus similar to $\Lambda$ ($\breve{\Lambda} = T^{-1}\Lambda T$ with $ T = WS^{-1} = \frac{1}{2\sqrt{2}\sqrt{z^2+1}}
\left[\begin{smallmatrix} 
	\sqrt{z^2+1}-z +\imath & \sqrt{z^2+1}+z-\imath\\ 2z & -2z\end{smallmatrix}\right]$). Please note that $L$ in \eqref{eq:S} is essentially  the same as $L$ in \eqref{eq:L}. 
Choosing
\[
	\label{eq:breveB}
\breve{\beta} = \begin{bmatrix} 2\Real(\mu)\\2\Real(\mu)\end{bmatrix}
\]
we find that  $\breve{\Lambda},\, \breve{\beta}$ satisfy \eqref{eq:famous_m}. Due to the choice of $\mu,$  \eqref{eq:eigcond} and \eqref{eq:sim_tabs} also hold. Thus the methods based on $\Lambda,\, \beta$ and on $\breve{\Lambda},\, \breve{\beta}$ are equivalent.

Finally, please note, that given $\Lambda,\, \beta$ satisfying \eqref{eq:famous_m} and $\check{\Lambda} = U^{-1}\Lambda U$ with a regular matrix $U \in \mathbb{C}^{2 \times 2}$ the condition
$0=\diag(\beta)\overline{\Lambda} + \Lambda^{\mathsf{T}}\diag(\beta)-\beta\beta^{\mathsf{T}}$
can not be used to determine whether or not $\breve{\beta}$ exists such that \eqref{eq:famous_m} is satisfied for $\breve{\Lambda}, \breve{\beta}.$ We have
\begin{align}
0&=\left(U^\mathsf{T}\diag(\beta)\overline{U}\right) \left(\overline{U}^{-1}\overline{\Lambda}\overline{U}\right) + \left(U^\mathsf{T}\Lambda^{\mathsf{T}}U^\mathsf{-T}\right)
\left(U^\mathsf{T}\diag(\beta)\overline{U}\right)-U^\mathsf{T}\beta\beta^{\mathsf{T}}\overline{U}\\ &=
Y \overline{\check{\Lambda}} +  \check{\Lambda}^\mathsf{T} Y -U^\mathsf{T}\beta\beta^{\mathsf{T}}\overline{U}
\end{align}
where $ Y = U^\mathsf{T}\diag(\beta)\overline{U}.$ Unfortunately, in general $Y$ will not be diagonal (this is readily checked for
$U = V$ and $\check{\Lambda} = \widetilde{\Lambda}$ or $U = T$ and $\check{\Lambda} = \breve{\Lambda}$). Moreover,
in general, $U^\mathsf{T}\beta \neq \beta^\mathsf{T}\overline{U}.$ Thus, considering the above transformation of 
 \eqref{eq:famous_m} does not help in order to determine an appropriate $\check{\beta}$ such that $\check{\Lambda}, \check{\beta}$ satisfy  \eqref{eq:famous_m}. 
One needs to check $0=\diag(\check{\beta})\overline{\check{\Lambda}} + \check{\Lambda}^{\mathsf{T}}\diag(\check{\beta})-\check{\beta}\check{\beta}^{\mathsf{T}}$ directly.

\subsection{Choice of parameters in \Cref{alg:quad_adi}}
\label{sec:geometric}
Here we discuss the choice of the parameters $\mu_1, \dots, \mu_N$ to obtain a good approximation to the Gramian. We know that the Lyapunov residual for iterates which fulfill the invariant \eqref{eq:invariant_complex} is given by $h_Nh_N^{\mathsf{H}}$. Hence for a small residual the norm $\|h_N\|$ has to be small.

Consider \Cref{alg:quad_adi}, that is, $R_{(j)}(z) = \frac{1+\overline{\mu_j}z}{1-\mu_jz}$.
Then with \eqref{eq:hMh} and \eqref{eq:iter_mat_diag} we find
\begin{align}
	h_N &= \left(\prod_{j=1}^{N}M_j\right) b \\
	& = V\left(\prod_{j=1}^N\begin{bmatrix}R_{(j)}(\lambda_1)& &\\&\ddots &\\ &&R_{(j)}(\lambda_n)\end{bmatrix}\right) V^{-1}b\\
	& = V\left(\prod_{j=1}^N\begin{bmatrix}\frac{1+\overline{\mu_j}\lambda_1}{1-\mu_j\lambda_1}& &\\&\ddots &\\ &&\frac{1+\overline{\mu_j}\lambda_n}{1-\mu_j\lambda_n}\end{bmatrix}\right) V^{-1}b.
\end{align}
Taking the 2-norm this implies
\begin{align}
	\| h_N \| \leq \|V\| \max_{i=1, \ldots, n}\left( \prod_{j=1}^{N} \frac{|1+\overline{\mu_j}\lambda_i|}{|1-\mu_j\lambda_i|} \right) \|V^{-1}\| \|b\|.
\end{align}
To minimize this error bound the parameters $\mu_j$ have to be chosen such that the problem
\begin{align}
	\min_{\mu_1, \ldots, \mu_N}\max_{i=1, \ldots, n}\prod_{j=1}^{N} \frac{|1+\overline{\mu_j}\lambda_i|}{|1-\mu_j\lambda_i|} 
\end{align}
is solved. It is equivalent to the rational min-max problem \cite[Sec.\,2.2]{benkuebt2013}. For $N=n$ the choice $\mu_j=-\overline{\lambda_j}^{-1}$, $j=1, \dots, N$ yields $h_N=0$ and thus $P_N=\mathcal{P}$. However,
in that case the final iterate $Z_N$ of \Cref{alg:quad_adi}  is of size $n \times n.$ Thus, this would not be a low rank approximation. Still this suggests that the parameters should somehow approximate the negative conjugated inverse of the eigenvalues $\lambda_i$ of $A$, i.e. $\mu_j\approx -\overline{\lambda_i}^{-1}$.

As we consider a stable system matrix $A$, all eigenvalues $\lambda_i$ have negative real part. Therefore, to obtain a factor with modulus smaller one, i.e. $\frac{|1+\overline{\mu_j}\lambda_i|}{|1-\mu_j\lambda_i|}<1$, the parameters $\mu_j$ must have positive real parts.

As $\omega_j=1$ in \Cref{alg:quad_adi} and due to \eqref{eq:eigcond}, the choice of $\mu_j$ with positive real parts guarantees the uniqueness of the solution $K$ of \eqref{eq:Kvec}. As \Cref{alg:quad_adi} is equivalent to the ADI iteration, we refer the reader to \cite{benkuebt2013} for a discussion of different strategies for choosing the parameters.

\subsection{Realification}\label{sec37}
\label{sec:realification}
We show that when using a proper set of shifts in \Cref{alg:quad_adi} we can avoid complex arithmetic in \Cref{alg:quad_adi} forced by the use of complex parameters $\mu_j.$ We follow the idea discussed in \Cref{sec:adi} and combine two (complex) iteration steps of the algorithm to a real one. Thereto \Cref{thm:dim1_tableaus} and the discussion thereafter is utilized. That is instead of two steps with the $1$-stage tableaus $\Lambda^{(j)}=\mu,\, \beta^{(j)}=2\Real(\mu)$ and $\Lambda^{(j+1)}=\overline\mu,\, \beta^{(j+1)}=2\Real(\mu)$ for $\mu\in \mathbb{C}_{+}$ only one step of \Cref{alg:quad} is performed with a suitable real $2$-stage Butcher tableau with $\Lambda,\, \beta$  such that $\sigma(\Lambda) = \{\mu,\,\overline\mu\}$ and meeting the requirements of \Cref{thm:dim1_tableaus}. As discussed in \Cref{sec36}  the real $2$-stage tableau with $\breve{\Lambda}$ and $\breve{\beta}$ as in  \eqref{eq:breveL} and \eqref{eq:breveB} can be used here. We will do so in the following discussion.

In order to explain our realification idea, we will make use of the iteration based on 
\begin{equation}\label{eq:xx}
\mathcal{H}_j = [h_{j-1}, h_{j-1}] + A\mathcal{H}_j\breve{\Lambda}^{\mathsf{T}}
\end{equation}
as in \eqref{def:KAHH}.
The ideas transfer to the iteration based on $K_j$ from \eqref{def:KAH} which is used in \Cref{alg:quad}
easily. We further assume that all previous iterates are real valued, in particular, $h_{j-1}\in \mathbb{R}^{n}$ and $P_{j-1} \in \mathbb{R}^{n \times n}.$

The iterate $\mathcal{H}_j=[\mathfrak{h}_1^{(j)},\mathfrak{h}_2^{(j)}]$  can be obtained from \eqref{eq:xx} via vectorization
	\begin{align}\label{eq:2nreal}
		\begin{bmatrix}
		I_n-\Real(\mu) A & -(\Real(\mu)+\varphi|\mu |) A\\
		-(\Real(\mu)-\varphi|\mu|)A & I_n-\Real(\mu) A
	\end{bmatrix}\begin{bmatrix}\mathfrak{h}_1^{(j)}\\\mathfrak{h}_2^{(j)}\end{bmatrix} &= \begin{bmatrix}h_{j-1}\\ h_{j-1}\end{bmatrix}.
	\end{align}
Thus, instead of solving two complex $n \times n$ systems in \Cref{alg:quad_adi} for the steps with $\Lambda^{(j)}=\mu$ and $\Lambda^{(j+1)}=\overline\mu$ it is possible to solve one real $2n \times 2n$ system \eqref{eq:2nreal}. As the system
matrix as well as the right hand side is real,  this gives $\mathcal{H}_j \in \mathbb{R}^{n\times 2}.$
	This realification approach is comparable to the approaches $\mathsf{M2^\ast}$ and $\mathsf{M4^\ast}$ presented in \cite[Chp. 4.1.5]{Kue16}. However, it was shown in \cite{Kue16} that these approaches based on $2n \times 2n$ real linear systems are mostly outperformed by the approach $\mathsf{M4}$ which is based on the solution of one complex $n \times n$ system of linear equations. We summarized this approach in \Cref{sec:adi}.

We continue our discussion with a modification of our approach such that instead of the $2n \times 2n$ real system \eqref{eq:2nreal} the solution of just one complex $n \times n$ system of linear equations is employed. The idea is to compose the iterate $\mathcal{H}_j=[\mathfrak{h}_1^{(j)},\mathfrak{h}_2^{(j)}]$ of the real and imaginary part of the solution of this complex system. For this purpose we intend to use the connection between a complex system of linear equations and its augmented real version, following the arguments in \cite{Day2001}. Consider real vectors $v_1,v_2,y_1,y_2\in \mathbb{R}^{n}$. Then the $n \times n$ complex system
\begin{align}
	(I_n-\mu A)(v_1+\imath v_2) = y_1 + \imath y_2
	\label{eq:equiv_complex_system}
\end{align}
is equivalent to the $2n \times 2n$ real system
\begin{align}
	\begin{bmatrix}
		I_n-\Real(\mu) A & \Imag(\mu) A\\
		-\Imag(\mu) A & I_n-\Real(\mu) A
	\end{bmatrix}\begin{bmatrix}v_1\\v_2\end{bmatrix} &= \begin{bmatrix}y_1\\ y_2\end{bmatrix}.
	\label{eq:Hdash2}
\end{align}
We now want to find a Butcher tableau
leading to an iteration which is equivalent to the one with $\breve{\Lambda}$ and $\breve{\beta}$ and for which the equivalence of \eqref{eq:equiv_complex_system} and \eqref{eq:Hdash2} can be used.
Clearly, $\widehat{\Lambda} \in \mathbb{R}^{2\times 2}$ from \eqref{eq:lambdahat} will lead to \eqref{eq:Hdash2} when $\mathcal{H}_j = [h_{j-1}, h_{j-1}] + A\mathcal{H}_j\widehat{\Lambda}^{\mathsf{T}}$ \eqref{def:KAHH} is vectorized.
But as shown in \Cref{sec36}, there is no appropriate $\widehat\beta$ such that the requirements of \Cref{thm:dim1_tableaus} are met. Thus, we cannot claim that this is equivalent to the iteration with $\breve\Lambda$.

However, from \eqref{eq:xx} we obtain with  $\widehat{\Lambda}=S^{-1}\breve{\Lambda} S$ and $S$ as in \eqref{eq:S}
\begin{align}
	\mathcal{H}_j' &= [h_{j-1}, h_{j-1}]S^{-\mathsf{T}} + A\mathcal{H}_j'{\widehat{\Lambda}}^{\mathsf{T}}
	\label{eq:Hdash}
\end{align}
for $\mathcal{H}_j'=[v_1,v_2]\in \mathbb{R}^{n\times 2}$ defined by $\mathcal{H}_j=\mathcal{H}_j'S^{\mathsf{T}}.$ The vectorization of \eqref{eq:Hdash} gives the $2n \times 2n$ real linear system \eqref{eq:Hdash2} with $[y_1,y_2]=[h_{j-1},h_{j-1}]S^{-\mathsf{T}} \in \mathbb{R}^{n \times 2}$, which is equivalent to the $n \times n$ complex system \eqref{eq:equiv_complex_system}.  Thus it is suggested to solve \eqref{eq:equiv_complex_system} for the right-hand side defined by $[y_1,y_2]=[h_{j-1},h_{j-1}]S^{-\mathsf{T}}$ in order to obtain $\mathcal{H}_j=[v_1,v_2]S^{\mathsf{T}} \in \mathbb{R}^{n \times 2}.$

Our discussion in the last paragraph is based on the choice $\breve{\Lambda} \in \mathbb{R}^{2 \times 2}$ as in \eqref{eq:breveL} and the fact that it is similar to $\widehat{\Lambda}.$ Alternatively, any other Butcher tableau with $\breve{\Lambda}, \breve{\beta}$ satisfying \eqref{eq:famous_m}, \eqref{eq:eigcond} and \eqref{eq:sim_tabs} and with $\breve{\Lambda}$ similar to $\widetilde{\Lambda}$ can be chosen. This leads to other similarity transformations than $S$ to obtain $\widehat\Lambda$ and thus produces different realifications.

We conclude our discussion by taking a closer look at our approach with $\breve\Lambda$ from \eqref{eq:breveL}. 
As for $S$ as in \eqref{eq:S} $[y_1,y_2]=[h_{j-1},h_{j-1}]S^{-\mathsf{T}}=[h_{j-1},0]$ holds,  the complex system \eqref{eq:equiv_complex_system} to solve reduces to
\begin{align}
	(I_n-\mu A)(v_1+\imath v_2) = h_{j-1}.
\end{align}
This system is just
\eqref{def:KAHH} for the choice $\Lambda^{(j)}=\mu.$ It is also the same as step 3 in \Cref{alg:adi32} for $\alpha_j = -\mu^{-1}.$
Next $\mathcal{H}_j=[v_1,v_2]S^{\mathsf{T}}$ has to be calculated for $S^\mathsf{T} = LQ$ as in \eqref{eq:S}
(see also \eqref{eq:L}).
We note that multiplying the iterate $\mathcal{H}_j$ from the right with the orthogonal matrix $Q\in \mathbb{R}^{2\times 2}$ does not  alter the approximation $P_j,$
	\begin{align}
		(\mathcal{H}_jQ)\diag(\breve\beta)(\mathcal{H}_j Q)^{\mathsf{H}} = \mathcal{H}_j \diag(\breve\beta) \mathcal{H}_j^{\mathsf{H}},
		\label{eq:orth_inv}
	\end{align}
as $\diag(\breve\beta)=2\Real(\mu)I_2.$ Therefore due to \eqref{eq:RK_matmat} and \eqref{eq:ZZH} we obtain the next approximate Cholesky factor $Z_j$ by appending
\begin{align}
	\sqrt{2\Real(\mu)}\mathcal{H}_jQ^{\mathsf{T}} &= \sqrt{2\Real(\mu)}\mathcal{H}_j'S^{\mathsf{T}}Q^{\mathsf{T}}=\sqrt{2\Real(\mu)}\mathcal{H}_j'L Q Q^{\mathsf{T}}\\
	&= \sqrt{2\Real(\mu)}\mathcal{H}_j'L =\sqrt{2\Real(\mu)}[v_1, v_2]L\\
	&= \sqrt{2\Real(\mu)}\sqrt{2}\left[v_1 + \frac{\Real(\mu)}{\Imag(\mu)}v_2,\, \sqrt{\frac{\Real(\mu)^2}{\Imag(\mu)^2}+1}\cdot v_2\right]
\end{align}
to $Z_{j-1}$. This assembly of the real and imaginary part has the same structure as in \eqref{eq:pad}.

All in all we have presented two approaches to keep the iterates real. The former one directly uses \Cref{thm:dim1_tableaus} to replace two $1$-stage complex tableaus by one real $2$-stage tableau. The latter one exploits the connection between a complex system of linear equations and its augmented real version, employing similarity transformations.

\section{Sylvester equation}
\label{sec:sylvester}
In this section we consider the Sylvester equation
\begin{align}
	A\mathcal{Y} - \mathcal{Y}B = FG^{\mathsf{T}}
	\label{eq:sylvester}
\end{align}
with the system matrices $A\in \mathbb{R}^{n\times n}$, $B\in \mathbb{R}^{m\times m}$, $F\in \mathbb{R}^{n\times r}$ and $G\in \mathbb{R}^{m\times r}$, $r\leq n,m$. We assume that the spectra of $A$ and $B$ are disjoint, $\sigma(A) \cap \sigma(B) = \emptyset,$ as then \eqref{eq:sylvester} has a unique solution.

Similar to the Lyapunov case where we only considered Lyapunov equations $A\mathcal{P}+\mathcal{P}A^{\mathsf{T}} = -bb^\mathsf{T}$ 
with rank-one right-hand side, we will restrict our discussion to the case $r=1$, $F =f$, $G = g$  here. The case $r>1$ with $F=[f_1, \cdots, f_r]$ and $G=[g_1, \cdots, g_r]$ can be reduced to $\mathcal{Y}=\sum_{i=1}^r\mathcal{Y}_i$ with
$A\mathcal{Y}_i - \mathcal{Y}_iB = f_i g_i^{\mathsf{T}}, i = 1,\, \ldots, r.$

In analogy to \eqref{eq:gramian_ode} we consider the system of ODEs
\begin{align}
	\frac{\mathrm d}{\mathrm dt}Y(t) &= \hat{h}(t)\breve{h}(t)^{\mathsf{T}}, & Y(0)&=0, \nonumber\\
	\frac{\mathrm d}{\mathrm dt}\hat{h}(t) &= A\hat{h}(t), & \hat{h}(0)&=f,\label{eq:ode_sylvester}\\
	\frac{\mathrm d}{\mathrm dt}\breve{h}(t) &= -B^{\mathsf{T}}\breve{h}(t), & \breve{h}(0)&=-g. \nonumber
	\end{align}
Please note that for $t \rightarrow \infty$ in general we will have  $Y(t) \not\rightarrow \mathcal{Y}.$ Nonetheless, as we will see,
\eqref{eq:ode_sylvester} is useful in order to derive approximations to $\mathcal{Y}.$

Due to the product rule the second derivative of $Y$ satisfies
\begin{align}
	\frac{\mathrm d^2}{\mathrm dt^2}Y(t) &= \frac{\mathrm d}{\mathrm dt}\hat{h}(t)\breve{h}(t)^{\mathsf{T}}\\
	&= A\hat{h}(t)\breve{h}(t)^{\mathsf{T}} - \hat{h}(t)\breve{h}(t)^{\mathsf{T}}B\\
	&= A\left(\frac{\mathrm d}{\mathrm dt}Y(t)\right) - \left(\frac{\mathrm d}{\mathrm dt}Y(t)\right)B.
\end{align}
By integrating both sides over the interval $[0,t]$ we find
\begin{align}
	\frac{\mathrm d}{\mathrm dt}Y(t) &= AY(t) - Y(t)B - fg^{\mathsf{T}}.
\end{align}
Thus the solution of the system of ODEs \eqref{eq:ode_sylvester} does not satisfy \eqref{eq:sylvester} exactly, a rank one residual does remain
\begin{align}
	AY(t)-Y(t)B - fg^{\mathsf{T}}=\hat{h}(t)\breve{h}(t)^{\mathsf{T}}.
	\label{eq:time_sylvester}
\end{align}

\subsection{Approximating $\mathcal{Y}$ by Runge-Kutta methods}
Unlike in the previous section, we propose here to solve the partitioned system of ODEs \eqref{eq:ode_sylvester} in two main steps. First, the latter two equations for $\hat{h}(t)$ and $\breve{h}(t)$ are solved, then the equation for $Y(t)$ is solved. Each of the three ODEs will be solved by a different method. We will make use of three different $s$-stage Runge-Kutta methods: The function $Y(t)$ is approximated using the Butcher tableau with $\Lambda^{(j)}\in \mathbb{C}^{s\times s}$ and $\beta^{(j)}\in \mathbb{C}^{s}$, the function $\hat{h}(t)$  with $\hat{\Lambda}^{(j)}\in \mathbb{C}^{s\times s}$ and $\hat{\beta}^{(j)}\in \mathbb{C}^{s}$ and the function $\breve{h}(t)$ with $\breve{\Lambda}^{(j)}\in \mathbb{C}^{s\times s}$ and $\breve{\beta}^{(j)}\in \mathbb{C}^{s}$.
As before, the upper index $j$ is used to denote the $j$th step of the iteration and w.l.o.g. the time step sizes are all set to $\omega_j=1$. 

In analogy to the derivation in the Lyapunov case we find that the application of the Runge-Kutta methods for
$\hat{h}(t)$ and $\breve{h}(t)$  lead to the iterations
\begin{align}
	\label{eq:hxks_iter}
	\hat{h}_j &= \hat{h}_{j-1} + \hat{K}_j\hat{\beta}^{(j)},\\
	\label{eq:bxks_iter}
	\breve{h}_j &= \breve{h}_{j-1} + \breve{K}_j\breve{\beta}^{(j)},
\end{align}
with $\hat{h}_0 = f$,  $\breve{h}_0 = -g$ (see \eqref{eq:RK_matmat}) and
\begin{align}
	\label{eq:hatK}
	\hat{K}_j &= \left[ A\hat{h}_{j-1}, \dots, A\hat{h}_{j-1} \right] + A\hat{K}_{j}(\hat{\Lambda}^{(j)})^{\mathsf{T}},\\
	\label{eq:breK}
	\breve{K}_j &= \left[ -B^{\mathsf{T}}\breve{h}_{j-1}, \dots, -B^{\mathsf{T}}\breve{h}_{j-1} \right] - B^{\mathsf{T}}\breve{K}_{j}(\breve{\Lambda}^{(j)})^{\mathsf{T}},
\end{align}
(see \eqref{def:KAH}) where $\hat{K}_j=A\hat{\mathcal{H}}_j$ and additionally $\breve{K}_j=-B^{\mathsf{T}}\breve{\mathcal{H}}_j$ holds, with 
\begin{align}
	\hat{\mathcal{H}}_j &= \left[ \hat{h}_{j-1}, \dots, \hat{h}_{j-1} \right] + \hat{K}_{j}(\hat{\Lambda}^{(j)})^{\mathsf{T}},\\
	\breve{\mathcal{H}}_j &= \left[ \breve{h}_{j-1}, \dots, \breve{h}_{j-1} \right] + \breve{K}_{j}(\breve{\Lambda}^{(j)})^{\mathsf{T}}
\end{align}
(see \eqref{def:KAHH}).
The solution of \eqref{eq:hatK} is unique if and only if
\begin{align}
	\hat\mu_p \neq \hat\lambda_q^{-1}, 
\end{align}
for all $\hat\mu_p\in\sigma(\hat\Lambda^{(j)})$ and all $\hat\lambda_q\in \sigma(A)$, while the solution of \eqref{eq:breK} is unique if and only if
\begin{align}
\breve\mu_k \neq -\breve\lambda_\ell^{-1}
\end{align}
for all  $\breve\mu_k \in \sigma(\breve\Lambda^{(j)})$ and all  $\breve\lambda_\ell \in \sigma(B)$ (see \eqref{eq:eigcond}).

The $j$th iterate $Y_j$ approximating the function $Y(t)$ is then given by
\begin{align}
	Y_j 
	&= Y_{j-1} + \hat{\mathcal{H}}_j \diag(\beta^{(j)}) \breve{\mathcal{H}}_j^{\mathsf{H}}
\end{align}
where $Y_0=0$ (see the derivations leading to \eqref{eq:RK_matmat}, in particular note that as in \eqref{eq:ohnektilde}, $\Lambda$ does not appear as the right hand side of $Y'(t) = \hat{h}(t)\breve{h}(t)^\mathsf{T}$ is independent of $Y(t)$).

Following the ideas from the previous section, we rewrite $Y_j$ in terms of two
low rank factors and a diagonal matrix
\[
Y_j = \hat{Z}_j\Gamma_j\breve{Z}_j^{\mathsf{H}}
\]
with
\[
	\hat{Z}_j = \left[ \hat{Z}_{j-1}, \hat{\mathcal{H}}_j \right],\quad
	\breve{Z}_j = \left[ \breve{Z}_{j-1}, \breve{\mathcal{H}}_j \right],\quad
	\Gamma_j = \diag(\Gamma_{j-1}, \beta^{(j)}).
\]
The solution of the Sylvester equation is neither symmetric nor positive definite, so the factors $\hat{Z}_j$ and $\breve{Z}_j$ need not be equal and $\Gamma_j$ may contain arbitrary complex valued entries.

Employing our main idea form the previous section. we will only consider Runge-Kutta methods whose iterates preserve the low rank property of the Sylvester residual, that is
\begin{align}
	AY_j-Y_jB - fg^{\mathsf{T}}=\hat{h}_j\breve{h}_j^{\mathsf{T}}.
	\label{eq:sylvester_res_iter}
\end{align}
Thus, following the arguments in Section \ref{sec:geometric}, the iterates $Y_j$ will not necessarily approximate the function $Y(t)$ very well, but $Y_j$ should be a good approximation to $\mathcal{Y}$ when the residual $\hat h_j\breve h_j^\mathsf{T}$ is small.

\subsection{Runge-Kutta methods which preserve \eqref{eq:sylvester_res_iter}}
To characterize the tableaus which ensure that all iterates fulfill \eqref{eq:sylvester_res_iter}, we insert the first iterates $Y_1, \hat{h}_1$ and $\breve{h}_1$ into \eqref{eq:sylvester_res_iter}. For the left hand side we obtain
\begin{align}
	AY_1-Y_1B - fg^{\mathsf{T}}&= A(Y_0 + \hat{\mathcal{H}}_1\diag(\beta^{(1)})\breve{\mathcal{H}}_1^{\mathsf{H}}) - (Y_0 + \hat{\mathcal{H}}_1\diag(\beta^{(1)})\breve{\mathcal{H}}_1^{\mathsf{H}})B - fg^{\mathsf{T}}\\
	&= A\hat{\mathcal{H}}_1\diag(\beta^{(1)})\breve{\mathcal{H}}_1^{\mathsf{H}} - \hat{\mathcal{H}}_1\diag(\beta^{(1)})\breve{\mathcal{H}}_1^{\mathsf{H}}B - fg^{\mathsf{T}}\\
	&= \hat{K}_1\diag(\beta^{(1)})\breve{\mathcal{H}}_1^{\mathsf{H}} + \hat{\mathcal{H}}_1\diag(\beta^{(1)})\breve{K}_1^{\mathsf{H}} - fg^{\mathsf{T}}\\
	&= \hat{K}_1\diag(\beta^{(1)})\left( [\breve{h}_0, \dots, \breve{h}_0]+\breve{K}_1(\breve{\Lambda}^{(1)})^{\mathsf{T}}\right)^{\mathsf{H}}\\
	&\phantom{=~~} +~ \left([\hat{h}_0, \dots, \hat{h}_0]+\hat{K}_1(\hat{\Lambda}^{(1)})^{\mathsf{T}}\right)\diag(\beta^{(1)})\breve{K}_1^{\mathsf{H}} - fg^{\mathsf{T}}\\
	&= \hat{K}_1(\beta^{(1)})\breve{h}_0^{\mathsf{H}} + \hat{K}_1\diag(\beta^{(1)})\overline{\breve{\Lambda}^{(1)}}\breve{K}_1^{\mathsf{H}}\\
	&\phantom{=~~} +~ \hat{h}_0(\beta^{(1)})^{\mathsf{T}}\breve{K}_1^{\mathsf{H}} +\hat{K}_1(\hat{\Lambda}^{(1)})^{\mathsf{T}}\diag(\beta^{(1)})\breve{K}_1^{\mathsf{H}} - fg^{\mathsf{T}},
\end{align}
while for the right hand side 
\begin{align}
	\hat{h}_1\breve{h}_1^{\mathsf{H}} &= \left(\hat{h}_{0} + \hat{K}_1\hat{\beta}^{(1)}\right) \left(\breve{h}_{0} + \breve{K}_1\breve{\beta}^{(1)}\right)^{\mathsf{H}}\\
	&= -fg^{\mathsf{H}} + \hat{h}_{0}(\breve{\beta}^{(1)})^{\mathsf{H}}\breve{K}_1^{\mathsf{H}}\\
	&\phantom{=~~} +~ \hat{K}_1\hat{\beta}^{(1)}\breve{h}_0^{\mathsf{H}} + \hat{K}_1\hat{\beta}^{(1)}(\breve{\beta}^{(1)})^{\mathsf{H}}\breve{K}_1^{\mathsf{H}}
\end{align}
holds. As $g$ is real $g^{\mathsf{T}}=g^{\mathsf{H}}$ and so the equation
\begin{align}
	&
\hat{K}_1\beta^{(1)}\breve{h}_0^{\mathsf{H}} + \hat{K}_1\diag(\beta^{(1)})\overline{\breve{\Lambda}^{(1)}}\breve{K}_1^{\mathsf{H}}
+ \hat{h}_0(\beta^{(1)})^{\mathsf{T}}\breve{K}_1^{\mathsf{H}} + \hat{K}_1(\hat{\Lambda}^{(1)})^{\mathsf{T}}\diag(\beta^{(1)})\breve{K}_1^{\mathsf{H}}\\
	&~~~~~~~~~~\stackrel{!}{=} \hat{h}_{0}(\breve{\beta}^{(1)})^{\mathsf{H}}\breve{K}_1^{\mathsf{H}} + \hat{K}_1\hat{\beta}^{(1)}\breve{h}_0^{\mathsf{H}} + \hat{K}_1\hat{\beta}^{(1)}(\breve{\beta}^{(1)})^{\mathsf{H}}\breve{K}_1^{\mathsf{H}}
\end{align}
has to hold in order to satisfy \eqref{eq:sylvester_res_iter}.
This implies
\begin{align}
	0 &= \hat{h}_0\left( \breve{\beta}^{(1)} - \overline{\beta^{(1)}} \right)^{\mathsf{H}}\breve{K}_1^{\mathsf{H}} + \hat{K}_1\left(\hat{\beta}^{(1)} - \beta^{(1)} \right)\breve{h}_0^{\mathsf{H}}\\
	&\phantom{=~~}+~\hat{K}_1\left( \hat{\beta}^{(1)}(\breve{\beta}^{(1)})^{\mathsf{H}} - \diag(\beta^{(1)})\overline{\breve{\Lambda}^{(1)}} - (\hat{\Lambda}^{(1)})^{\mathsf{T}}\diag(\beta^{(1)}) \right) \breve{K}_1^{\mathsf{H}},
\end{align}
which is satisfied for $\hat{\beta}^{(1)}={\beta}^{(1)},$ $\breve{\beta}^{(1)}=\overline{\beta^{(1)}}$, and
\begin{align}
	\diag(\beta^{(1)})\overline{\breve{\Lambda}^{(1)} }+ (\hat{\Lambda}^{(1)})^{\mathsf{T}}\diag(\beta^{(1)}) - \beta^{(1)}(\beta^{(1)})^{\mathsf{T}} = 0.
\end{align}
 Via induction this leads to the conditions
\begin{align}
	\label{eq:famous_m-sylvester}
\begin{split}
	0&=\diag(\beta^{(j)})\overline{\breve{\Lambda}^{(j)}} + (\hat{\Lambda}^{(j)})^{\mathsf{T}}\diag(\beta^{(j)}) - \beta^{(j)}(\beta^{(j)})^{\mathsf{T}},\\
	\hat{\beta}^{(j)}&={\beta}^{(j)},\\
	\breve{\beta}^{(j)}&=\overline{\beta^{(j)}}.
\end{split}
\end{align}
Please note that  $\breve{\Lambda}^{(j)}, \breve{\beta}^{(j)}$ and $\hat{\Lambda}^{(j)}, \hat{\beta}^{(j)}$ from the Butcher tableaus for approximating $\breve{h}(t)$ and $\hat{h}(t)$ are relevant here, as well as $\beta^{(j)}$ from the
Butcher tableau for $Y(t),$ but as mentioned above not the matrix $\Lambda^{(j)}.$
DIRK tableaus which fulfill \eqref{eq:famous_m-sylvester} are given  by
\begin{align}\label{eq:DIRK_sylvester}
	\hat\Lambda = \begin{bmatrix}
		\hat\mu_1& 0& \cdots & 0\\
		\beta_1 & \hat\mu_2& \cdots& 0\\
		\vdots & \vdots& \ddots&\vdots\\
		\beta_1&\beta_2&\cdots&\hat\mu_s
	\end{bmatrix},\
	\breve\Lambda = \begin{bmatrix}
		\breve\mu_1& 0& \cdots & 0\\
		\overline\beta_1 & \breve\mu_2& \cdots& 0\\
		\vdots & \vdots& \ddots&\vdots\\
		\overline\beta_1&\overline\beta_2&\cdots&\breve\mu_s
	\end{bmatrix},\
	\beta = \begin{bmatrix}
		\hat\mu_1+\overline{\breve\mu_1}\\
		\hat\mu_2+\overline{\breve\mu_2}\\
		\vdots\\
		\hat\mu_s+\overline{\breve\mu_s}\\
	\end{bmatrix}.
\end{align}

\subsection{Multiplicative update formulae for $\hat{h}_j$ and $\breve{h}_j$}
Let us assume that \eqref{eq:famous_m-sylvester} holds with $\beta^{(j)}_k \neq 0,\, k = 1, \ldots, s.$ Further assume that $A$ and $B$ are diagonalizable, that is, $A=\hat{V}\hat{D}\hat{V}^{-1}$ and $B=\breve{V}\breve{D}\breve{V}^{-1}$ with $\hat{D}=\diag(\hat{\lambda}_1, \dots, \hat{\lambda}_n)$ and $\breve{D}=\diag(\breve{\lambda}_1, \dots, \breve{\lambda}_m).$ Then, as in Section \ref{subsec_multupdate} we find that
the iterations \eqref{eq:hxks_iter} and \eqref{eq:bxks_iter} can be written in multiplicative form as in \eqref{eq:iter_mat_diag}
\begin{align}
	\label{eq:M_sylvester1}
	\hat{h}_j &= \hat{V}\begin{bmatrix}\hat{R}_j(\hat{\lambda}_1)& & &\\ &\hat{R}_j(\hat{\lambda}_2)&&\\&&\ddots &\\ &&&\hat{R}_j(\hat{\lambda}_n) \end{bmatrix}\hat{V}^{-1}\hat{h}_{j-1},\\
	\label{eq:M_sylvester2}
	\breve{h}_j &= \breve{V}^{-\mathsf{T}}\begin{bmatrix}\breve{R}_j(-\breve{\lambda}_1)& & &\\ &\breve{R}_j(-\breve{\lambda}_2)&&\\&&\ddots &\\ &&&\breve{R}_j(-\breve{\lambda}_m) \end{bmatrix}\breve{V}^{\mathsf{T}}\breve{h}_{j-1},
\end{align}
with the stability functions 
\begin{align}
	\hat{R}_j(z) &= 1 + z(\beta^{(j)})^{\mathsf{T}}(I-z\hat{\Lambda}^{(j)})^{-1}\mathds{1}_s,\\
	\breve{R}_j(z) &= 1 + z(\overline{\beta^{(j)}})^{\mathsf{T}}(I-z\breve{\Lambda}^{(j)})^{-1}\mathds{1}_s.
\end{align}
As in the Lyapunov case, these stability functions only depend on the eigenvalues of the tableaus $\hat{\Lambda}^{(j)}$ and $\breve{\Lambda}^{(j)}$. In order to see this, first observe that for $\beta^{(j)}_i \neq 0$ for $i =1, \ldots, s$
 we find from \eqref{eq:famous_m-sylvester} by multiplying with $\diag(\beta^{(j)})^{-1}$ from the right (respectively left)
\begin{align}
	0 &= \diag(\beta^{(j)})\overline{\breve{\Lambda}^{(j)}}\diag(\beta^{(j)})^{-1} + (\hat{\Lambda}^{(j)})^{\mathsf{T}} - \beta^{(j)}\mathds{1}_s^{\mathsf{T}},\\
	0 &= \overline{\breve{\Lambda}^{(j)}} + \diag(\beta^{(j)})^{-1}(\hat{\Lambda}^{(j)})^{\mathsf{T}}\diag(\beta^{(j)}) - \mathds{1}_s(\beta^{(j)})^{\mathsf{T}}.
\end{align}
As any matrix is similar to its transpose, these equations imply the similarities
\begin{align}
	\label{eq:similar_sylvester}
\begin{split}
	\hat{\Lambda}^{(j)}-\mathds{1}_s(\beta^{(j)})^{\mathsf{T}} &\quad\sim\quad -\overline{\breve{\Lambda}^{(j)}},\\
	\breve{\Lambda}^{(j)}-\mathds{1}_s(\beta^{(j)})^{\mathsf{H}} &\quad\sim\quad -\overline{\hat{\Lambda}^{(j)}}.
\end{split}
\end{align}

Let $\hat{\Lambda}^{(j)}$ have eigenvalues $\hat{\mu}_1, \dots, \hat{\mu}_s$ and let $\breve{\Lambda}^{(j)}$ have eigenvalues $\breve{\mu}_1, \dots, \breve{\mu}_s$. Now with \eqref{eq:stab} and \eqref{eq:similar_sylvester} we have (similar to \eqref{eq:char_poly_zinv})
\begin{align}\label{eq:stab_sylv}
\begin{split}
	\hat{R}_j(z) &= \frac{\det(I-z(\hat{\Lambda}^{(j)}-\mathds{1}_s(\beta^{(j)})^{\mathsf{T}}))}{\det(I-z\hat{\Lambda}^{(j)})}
	= \frac{\det(I+z\overline{\breve{\Lambda}^{(j)}})}{\det(I-z\hat{\Lambda}^{(j)})}
=\prod_{i=1}^{s}\frac{(1+z\overline{\breve{\mu}_i})}{(1-z\hat{\mu}_i)},\\
	\breve{R}_j(z) &= \frac{\det(I-z(\breve{\Lambda}^{(j)}-\mathds{1}_s(\beta^{(j)})^{\mathsf{H}}))}{\det(I-z\breve{\Lambda}^{(j)})}
	= \frac{\det(I+z\overline{\hat{\Lambda}^{(j)}})}{\det(I-z\breve{\Lambda}^{(j)})}
=\prod_{i=1}^{s}\frac{(1+z\overline{\hat{\mu}_i})}{(1-z\breve{\mu}_i)}.
\end{split}
\end{align}
Thus, if \eqref{eq:sylvester_res_iter} is enforced, then the stability functions of the Runge-Kutta method for $\hat{h}(t)$ and for $\breve{h}(t)$ are not independent of each other. Both depend on the eigenvalues of $\hat\Lambda^{(j)}$ and $\breve\Lambda^{(j)}$. Thus, also the iterates
$\hat{h}_j$ \eqref{eq:M_sylvester1} and $\breve{h}_j$ \eqref{eq:M_sylvester2} depend on those eigenvalues.

The stability functions corresponding to the $s$-stage DIRK tableaus as in \eqref{eq:DIRK_sylvester} will have the form \eqref{eq:stab_sylv}.
As in the previous section, we can restrict ourselves to the use of several different $1$-stage Butcher tableaus satisfying \eqref{eq:famous_m-sylvester}, i.e. $\hat\Lambda^{(j)}=\hat\mu_j\in\mathbb{C}$ and $\breve\Lambda^{(j)}=\breve\mu_j\in\mathbb{C}$ with 
 $\beta^{(j)}=\hat{\beta}^{(j)}=\hat{\mu}_j + \overline{\breve{\mu}_j} \in \mathbb{C}$ and $\overline{\beta^{(j)}}=\breve{\beta}^{(j)}=\overline{\hat{\mu}_j} + \breve{\mu}_j \in \mathbb{C}$.

The resulting iteration for a low rank approximation to the solution of the Sylvester equation \eqref{eq:sylvester} is summarized in \Cref{alg:lr_sylvester}. It is equivalent to \cite[Alg. 3.4]{Kue16} with $\alpha_j=-\breve\mu_j^{-1}$ and $\beta_j=\hat\mu_j^{-1}$ (and $E=I_n$, $C=I_m$, $r=1$), in analogy to \Cref{thm:algo_equiv}.

In the special case $B=-A^{\mathsf{T}}$ and $g=-f =b$ the Sylvester equation \eqref{eq:sylvester} reduces to the Lyapunov equation \eqref{eq:lyap}. If additionally for the parameters $\hat\mu_j=\breve\mu_j,\, j=1, \dots, N,$ holds, then \Cref{alg:lr_sylvester} reduces to \Cref{alg:quad_adi} and we find $\hat\mu_j+\overline{\breve\mu}_j=2\Real(\hat\mu_j)\in \mathbb{R}$.

\begin{algorithm}[ht]
    \caption{Low rank solution to \eqref{eq:sylvester} via $1$-stage Runge-Kutta methods}
    \label{alg:lr_sylvester}
    \begin{algorithmic}[1] 
	    \Input $A\in \mathbb{R}^{n\times n}$, $B\in \mathbb{R}^{m\times m}$, $f\in\mathbb{R}^{n\times 1}$, $g\in \mathbb{R}^{m\times 1}$, parameters $\left\{ \hat\mu_1,\dots,\hat\mu_{N} \right\}\subset \mathbb{C}$ and $\left\{ \breve\mu_1,\dots,\breve\mu_{N} \right\}\subset \mathbb{C}$
	    \Output  matrices $\hat Z\in \mathbb{C}^{n\times N}$, $\breve Z\in \mathbb{C}^{m\times N}$ and $\Gamma\in \mathbb{C}^{N\times N}$ with $\hat{Z} \Gamma \breve{Z}^{\mathsf{H}}\approx \mathcal{Y}$
	    \State initialize $\hat{h}_0=f$, $\breve{h}_0=-g$, $\hat{Z}_0 = \breve{Z}_0 = \Gamma_0=[\ ]$
    \For {$j=1,\dots,N$}
    \State solve $(I-\hat\mu_j A)\hat K_j=A\hat h_{j-1}$ for $\hat K_j$
    \State solve $(I+\breve\mu_j B^{\mathsf{T}})\breve K_j=-B^{\mathsf{T}}\breve h_{j-1}$ for $\breve K_j$
    \State $\hat{\mathcal{H}}_j=\hat h_{j-1} + \hat \mu_j \hat K_j$
    \State $\breve{\mathcal{H}}_j=\breve h_{j-1} + \breve \mu_j \breve K_j$
    \State update  $\hat Z_j=[\hat Z_{j-1}, \hat{\mathcal{H}}_j]$
    \State update  $\breve Z_j=[\breve Z_{j-1}, \breve{\mathcal{H}}_j]$
    \State update  $\Gamma_j = \diag(\Gamma_{j-1}, (\hat\mu_j+\overline{\breve{\mu}}_j)I_r)$
    \State $\hat h_{j} = \hat h_{j-1} +(\hat\mu_j+\overline{\breve{\mu}}_j) \hat K_j$
    \State $\breve h_{j} = \breve h_{j-1} + (\overline{\hat\mu}_j+\breve{\mu}_j) \breve K_j$
    \EndFor
    \State $\hat Z=\hat Z_N$, $\breve Z=\breve Z_N$, $\Gamma=\Gamma_N$
    \end{algorithmic}
\end{algorithm}

\section{Conclusion}
\label{sec:conclusion}
In this paper we presented a new derivation of residual based low rank iterations for the solution of (large-scale) linear matrix equations. For the Lyapunov equation a quadrature approach is evident, as for a stable system matrix the solution has an integral representation \eqref{eq:P}. We applied a Runge-Kutta method parameterized by an arbitrary and possibly complex valued Butcher tableau to approximate the solution of the system of ODEs \eqref{eq:gramian_ode}, resulting in \Cref{alg:quad}. The Lyapunov residual which remains after each iteration step
is determined.  Runge-Kutta methods that preserve the rank of the initial residual were characterized. By making use of the stability function of a Runge-Kutta method it was shown that these methods are equivalent to DIRK methods and thus equivalent to \Cref{alg:quad_adi}, which is itself equivalent to the ADI iteration. A realification approach based on similarity transformations was presented.

All ideas are applied in slightly modified form to the Sylvester equation in \Cref{sec:sylvester}. Although the solution of the system of ODEs \eqref{eq:ode_sylvester} does not converge to the solution of the Sylvester equation, the application of Runge-Kutta methods which preserve the initial rank of the residual yields a sound approximation, justified by \eqref{eq:M_sylvester1} and \eqref{eq:M_sylvester2}. That is, by retaining a geometric, qualitative property during the quadrature the ADI equivalent \Cref{alg:lr_sylvester} for the approximation of the solution of a Sylvester equation was derived from a system of ODEs.

\bibliographystyle{elsarticle-harv}
\bibliography{literatur}

\end{document}